\documentclass[a4paper,twoside,11pt,reqno]{memoir}
\usepackage{fullpage}
\usepackage[T1]{fontenc}
\usepackage[utf8]{inputenc}
\usepackage{amsthm}
\usepackage{hyperref}
\usepackage[slantedGreeks, partialup, noDcommand]{kpfonts}
\usepackage{graphicx}
\usepackage[mathscr]{euscript}
\usepackage{caption}
\usepackage{tikz}
\usepackage{braket}
\usepackage{mathrsfs,mathtools}

\usetikzlibrary{trees}

\hypersetup{ocgcolorlinks=true,allcolors=testc}
\hypersetup{
	colorlinks   = true,
	citecolor    = black
}
\hypersetup{linkcolor=black}
\hypersetup{urlcolor=black}

\newcommand{\ketbra}[2]{\ket{#1}\!\!\bra{#2}}

\newtheorem{theorem}{Theorem}
\newtheorem{lemma}[theorem]{Lemma}

\newtheorem{proposition}[theorem]{Proposition}

\newtheorem{remark}[theorem]{Remark}
\newtheorem{definition}[theorem]{Definition}

\newcommand{\C}{\mathbb{C}} 
\newcommand{\R}{\mathbb{R}} 
\newcommand{\N}{\mathbb{N}} 
\newcommand{\Z}{\mathbb{Z}} 
\newcommand{\e}{\epsilon}
\newcommand{\eps}{\epsilon}
\newcommand{\E}{\mathbb{E}}
\newcommand{\T}{\mathbb{T}} 
\newcommand{\cO}{\mathcal{O}}
\newcommand{\un}{\mathbf{1}}
\newcommand{\go}{K} 
\newcommand{\grpord}{K} 
\newcommand{\bh}{\textnormal{BH}}
\newcommand{\BH}{\textnormal{BH}}
\newcommand{\cS}{\mathcal{S}}
\renewcommand{\d}{\textnormal d} 

\newcommand{\cj}{\mathcal{J}}
\newcommand{\ba}{\mathbf{a}}

\newcommand{\bi}{\mathbf{i}}
\newcommand{\bj}{\mathbf{j}}
\newcommand{\bp}{\mathbf{p}}
\newcommand{\bq}{\mathbf{q}}
\newcommand{\br}{\mathbf{r}}
\renewcommand{\bs}{\mathbf{s}}
\newcommand{\tr}{\textnormal{tr}}
\newcommand{\om}{\omega}
\newcommand{\Om}{\Omega}
\newcommand{\al}{\alpha}
\newcommand{\conv}{\textnormal{conv}}
\newcommand{\supp}{\textnormal{supp}}
\newcommand{\iu}{i}

\newcommand{\D}{\mathbb{D}}	

\newcommand{\Lesssim}[1]{\lesssim_{{\textstyle\mathstrut}{#1}}}

\newcommand{\underbracedmatrix}[2]{%
	\left(\;
	\smash[b]{\underbrace{
			\begin{matrix}#1\end{matrix}
		}_{#2}}
	\;\right)
	\vphantom{\underbrace{\begin{matrix}#1\end{matrix}}_{#2}}
}

\newcommand{\Gtrsim}[1]{\gtrsim_{{\textstyle\mathstrut}{#1}}}

\title{Three lectures on Fourier analysis and learning theory}

\author{Haonan Zhang
	\thanks{These are expanded notes of three lectures given at IASM of HIT, Harbin, during summer 2024. The author is grateful for their warm hospitality.}
	\thanks{Department of Mathematics,
		University of South Carolina,
		Columbia, SC 29208, USA.}
	\thanks{
		Email: haonanzhangmath@gmail.com, haonanz@mailbox.sc.edu.}}

\begin{document}

	\maketitle
	
	\begin{abstract}
		Fourier analysis on the discrete hypercubes $\{-1,1\}^n$ has found numerous applications in learning theory. A recent breakthrough involves the use of a classical result from Fourier analysis, the Bohnenblust--Hille inequality, in the context of learning low-degree Boolean functions. In these lecture notes, we explore this line of research and discuss recent progress in discrete quantum systems and classical Fourier analysis. 
	\end{abstract}

	\bigskip
	
		
	
	\vspace{-1cm}
	\tableofcontents
	
	\chapter{Learning low-degree Boolean functions and Bohnenblust--Hille inequality}

	\section{Learning low-degree Boolean functions}
	
	The analysis of functions $f$ on discrete hypercubes $\{-1,1\}^n$ is of fundamental importance to theoretical computer science. For many applications, it suffices to consider \emph{Boolean functions} $f:\{-1,1\}^n\to \{-1,1\}$. For the Fourier analysis or functional analysis, $f$ can be real/complex-valued or even vector-valued. In the sequel, we denote by $\R$ and $\C$ the sets of real numbers and complex numbers, respectively. We use $\N$ to denote the set of positive integers.

	Any $f:\{-1,1\}^n\to \C$ has the \emph{Fourier--Walsh expansion}
	\begin{equation}
		f(x)=\sum_{S\subset [n]}\widehat{f}(S)\chi_S(x),\qquad x=(x_1,\dots, x_n)\in \{-1,1\}^n
	\end{equation}
	where for each $S\subset [n]:=\{1,\dots, n\}$, $\widehat{f}(S)\in \C$ is the Fourier coefficient, and $\chi_S$ is the character given by 
	\begin{equation}
		\chi_S(x):=\prod_{j\in S}x_j.
	\end{equation}
	For $d\in \N$, we say $f$ is \emph{of degree at most $d$} if $\widehat{f}(S)=0$ whenever $|S|>d$. The degree serves as an important complexity measure \cite{odonnell14book}. As usual, we equip $\{-1,1\}^n$ with the uniform probability measure  to define the $L^p$-norms and expectation $\E$, and we always consider the $\ell^p$-norms of the Fourier coefficients $(\widehat{f}(S))_S$ so that we have the \emph{Parseval's identity}
	\begin{equation}\label{eq:parseval boolean}
		\|\widehat{f}\|_{2}=\|f\|_2.
	\end{equation}

	One fundamental task in the theoretical computer science is to learn Boolean functions of low degree from its random queries \cite{odonnell14book}. Fix $\epsilon,\delta\in (0,1)$ and $d\ge 1$. Consider the class of low-degree Boolean  functions
	\begin{equation}
		\mathcal{F}^{\le d}_n:=\{f:\{-1,1\}^n\to \{-1,1\} \textnormal{ is of degree at most }d\}.
	\end{equation}
	One may also relax the constraint $|f|=1$ to $|f|\le 1$ in the following arguments.
	Let $N=N(\e,\delta, d,n)$ be the smallest positive integer such that for any $f\in \mathcal{F}^{\le d}_n$ and for any $N$ i.i.d. random variables $X_1,\dots, X_N$ that are uniformly distributed on $\{-1,1\}^n$, as well as the random queries 
	\begin{equation}
		(X_1, f(X_1)), \dots, (X_N, f(X_N)),
	\end{equation}
	one can construct a random function $h:\{-1,1\}^n\to\R$ that is close to the unknown $f$ with high probability 
	\begin{equation}
		\Pr\left(\|h-f\|_2^2\le \e\right)\ge 1-\delta.
	\end{equation}
	The question is: what is $N(\e,\delta, d,n)$? In this lecture, we focus on its dependence on the dimension $n$, and the answer is $\cO_{d,\e,\delta}(\log n)$ which is sharp \cite{ei22learning,eis23low}.
	
	This answer is surprising in the following sense. By Parseval's identity \eqref{eq:parseval boolean}, learning $f$ in $L^2$ is identical to learning $\widehat{f}$ in $\ell^2$. Knowing that $f$ is of degree at most $d$, the number of non-zero Fourier coefficients can be as large as 
	\begin{equation}\label{eq:number of nonzero fourier coefficients}
		\sum_{k\le d}\binom{n}{k}=\cO_d(n^d).
	\end{equation}
	A classical algorithm of Linial, Mansour and Nisan (LMN) \cite{lmn93} shows that $\cO_{d,\e,\delta}(n^d\log n)$ random queries suffices to learn $f$ in the above model, which we shall explain now.

	The idea of the LMN algorithm is very simple. Fix $b>0$ to be chosen later. For any $S\subset [n]$ with $|S|\le d$ we form the empirical Fourier coefficient
	\begin{equation}
		\alpha_S:=\frac{1}{N}\sum_{j=1}^{N}f(X_j)\chi_{S}(X_j).
	\end{equation} 
	By definition, this is the average of $N$ i.i.d. random variables having expectation $\E \alpha_S=\widehat{f}(S)$. So the Chernoff--Hoeffding inequality gives (we used the assumption that $|f|\le 1$ here)
	\begin{equation}\label{ineq:chernoff}
		\Pr\left(|\alpha_S-\widehat{f}(S)|> b\right)\le 2\exp\left(-Nb^2/2\right) .
	\end{equation}
	This, together with the union bound, yields
	\begin{equation}
		\Pr\left(|\alpha_S-\widehat{f}(S)|\le b\textnormal{ for all }|S|\le d\right)\ge 1-2\sum_{k=0}^{d}\binom{n}{k}\exp\left(-Nb^2/2\right) 
	\end{equation}
	which will be bounded from below by $1-\delta$ if we choose 
	\begin{equation}\label{eq:choice of N in terms of b}
		N=\left\lceil \frac{2}{b^2}\log\left(\frac{2}{\delta}\sum_{k=0}^{d}\binom{n}{k}\right)\right\rceil.
	\end{equation}
	Now we form the random function 
	\begin{equation}
		h=h_b:=\sum_{|S|\le d}\alpha_S\chi_S. 
	\end{equation}
	Then with probability at least $1-\delta$, we have 
	\begin{equation}
		\|h-f\|_2^2=\sum_{|S|\le d}|\alpha_S-\widehat{f}(S)|^2\le b^2\sum_{k=0}^{d}\binom{n}{k}.
	\end{equation}
	Setting the right-hand side to be $\e$ gives the value of $b$. Plugging this value of $b$ to \eqref{eq:choice of N in terms of b} yields 
	$N=\cO_{d,\e,\delta}(n^d\log n)$. 
	
	\medskip

	Based on the above discussion, it was believed for a long time that  $\cO_{d,\e,\delta}(n^d\log n)$ is best possible. However, this is far from the right answer $\cO_{d,\e,\delta}(\log n)$ as we mentioned earlier. In 2021, this bound was improved to $\cO_{d,\e,\delta}(n^{d-1}\log n)$ by Iyer, Rao, Reis, Rothvoss and Yehudayoff \cite{iyer21tight}. Later on, Eskenazis and Ivanisvili \cite{ei22learning} improved it to $\cO_{d,\e,\delta}(\log n)$, and the same authors and Streck \cite{eis23low} proved that $\cO(\log n)$ is best possible. The algorithm of Eskenazis and Ivanisvili is almost the same as that of LMN, with the only new ingredient a Fourier analysis inequality named after Bohnenblust and Hille that is the main focus of this lecture. Before going to this inequality, let us revisit the LMN algorithm and see where we lose. 
	
	Recall that $|f|\le 1$ and $\{-1,1\}^n$ is a probability measure space, so by Parseval's identity \eqref{eq:parseval boolean}
	\begin{equation}
		\sum_{|S|\le d}|\widehat{f}(S)|^2=\|f\|_2^2\le \|f\|_\infty^2\le 1,
	\end{equation}
	suggesting that many non-zero $\widehat{f}(S)$ are very small, say no more than  $\mathcal{O}(n^{-d})$ as $n\to \infty$ in view of \eqref{eq:number of nonzero fourier coefficients}. This is negligible compared to the parameter $b$ in \eqref{ineq:chernoff} as $n\to \infty$.
	
	In other words, the Fourier spectrum of $f$ consists of the influential part $\left\{|\widehat{f}(S)|>a\right\}$ and the negligible part $\left\{|\widehat{f}(S)|\le a\right\}$ with $a>0$ some threshold parameter. The cardinality of influential part is small regardless of $n$ by Parseval's identity \eqref{eq:parseval boolean} and Markov's inequality:
	\begin{equation}\label{ineq:number of influential part}
		|\{S:|\widehat{f}(S)|>a\}|\le \frac{1}{a^2}\|\widehat{f}\|_{2}^2=\frac{1}{a^2}\|f\|_{2}^2\le \frac{1}{a^2}\|f\|_{\infty}^2=\frac{1}{a^2}
	\end{equation}
	and we should keep these $\{\widehat{f}(S)\}$ to learn (though we are not able to decide if $|\widehat{f}(S)|>a$ using random queries, we may use $|\alpha(S)|$ as substitutes). As for the negligible part, we may certainly discard a small number of them at a low cost of accuracy. The issue is that their number can be as large as $\cO_d(n^d)$ in view of \eqref{eq:number of nonzero fourier coefficients} and \eqref{ineq:number of influential part}. Can we discard all of them at a small cost? This would be possible if we have the following estimate
	\begin{equation}
		\sum_{|\widehat{f}(S)|\le a}|\widehat{f}(S)|^2\le \sum_{|\widehat{f}(S)|\le a}|\widehat{f}(S)|^p a^{2-p}\stackrel{?}{\le} C(d)a^{2-p}
	\end{equation}
	where we need $p<2$ in the first inequality and that $C(d)$ is dimension-free in the second inequality. To summarize, we may discard a large number of negligible Fourier coefficients at an additional cost of $C(d)a^{2-p}$ in accuracy to improve the sample complexity significantly, provided the crucial estimate 
	\begin{equation}\label{ineq:crucial dimension free to discard negligible}
		\sum_{|\widehat{f}(S)|\le a}|\widehat{f}(S)|^p \stackrel{?}{\le} C(d)
	\end{equation}
	for some $p<2$ and for all Boolean (or bounded in general) $f$ of low degree.  The rest is to optimize over the threshold parameters $a,b>0$. The crucial inequality \eqref{ineq:crucial dimension free to discard negligible} was known:
	
	\begin{theorem}[Boolean Bohnenblust--Hille]
		Fix $d\ge 1$. There exists a constant $C(d)>0$ such that for all $n\ge 1$ and all $f:\{-1,1\}^n\to \R$ of degree at most $d$, we have
		\begin{equation}\label{ineq:boolean bh}
			\|\widehat{f}\|_{\frac{2d}{d+1}}\le C(d)\|f\|_\infty.
		\end{equation}
		Denoting the best constant by $\bh^{\le d}_{\{\pm 1\}}$. Then there exists a universal $C>0$ such that $\bh^{\le d}_{\{\pm 1\}}\le C^{\sqrt{d\log d}}$.
	\end{theorem}
	
	This inequality was first proved by Blei \cite{blei01}, and was recently revisited by Defant, Masty\l o and P\'erez \cite{dmp19} with a new method achieving this sub-exponential upper bound of $C^{\sqrt{d\log d}}$. We will come back to this upper bound later, but the optimal upper bound remains open. However, very recently, Arunachalam, Dutt, Gutiérrez and Palazuelos \cite{carlos24bh} obtained the optimal constant $C(d)=2^{\frac{d-1}{d}}$ when restricted to Boolean functions $f:\{-1,1\}^n\to \{-1,1\}$. 

	With this inequality at hand, we may finish the argument of Eskenazis and Ivanisvili \cite{ei22learning} proving the sufficiency of $\cO_{d,\e,\delta}(\log n)$. Starting from \eqref{eq:choice of N in terms of b}, we shall construct a random function $h$ with ``influential Fourier coefficients" only. Fix a threshold parameter $a>b$, and consider
	\begin{equation}
		\cS_a:=\{S:|S|\le d \textnormal{ and }|\alpha_S|\ge a\}.
	\end{equation}
	Continuing our discussion above, we know that with probability at least $1-\delta$, 
	\begin{equation}
		\begin{cases}
			|\widehat{f}(S)|\ge |\alpha_S|-|\alpha_S-\widehat{f}(S)|\ge a-b & S\in \cS_a\\
			|\widehat{f}(S)|\le |\alpha_S|+|\alpha_S-\widehat{f}(S)|\ge a+b & S\notin \cS_a
		\end{cases}.
	\end{equation}
	By Markov's inequality and Boolean Bohnenblust--Hille \eqref{ineq:boolean bh},
	\begin{equation}
		|\cS_a|\le (a-b)^{\frac{2d}{d+1}}\sum_{S\in \cS_a}|\widehat{f}(S)|^{\frac{2d}{d+1}}
		\le (a-b)^{\frac{2d}{d+1}}\left(\bh^{\le d}_{\{\pm 1\}}\right)^{\frac{2d}{d+1}}.
	\end{equation}
	Consider the random function 
	\begin{equation}
		h=h_{a,b}:=\sum_{S\in \cS_a}\alpha_S\chi_S.
	\end{equation}
	All combined, we have  with probability at least $1-\delta$, 
	\begin{align*}
		\|h-f\|_2^2
		&=\sum_{S\in \cS_a}|\widehat{f}(S)-\alpha_S|^2+\sum_{S\notin \cS_a}|\widehat{f}(S)|^2\\
		&\le |\cS_a|b^2+(a+b)^{\frac{2}{d+1}}\sum_{|S|\le d}|\widehat{f}(S)|^{\frac{2d}{d+1}}\\
		&\le \left(\bh^{\le d}_{\{\pm 1\}}\right)^{\frac{2d}{d+1}}\left((a-b)^{-\frac{2d}{d+1}}b^2+(a+b)^{\frac{2}{d+1}}\right).
	\end{align*}
	Choosing $a=b(1+\sqrt{d+1})$, we have 
	\begin{equation}
		\|h-f\|_2^2\le \left(\bh^{\le d}_{\{\pm 1\}}\right)^{\frac{2d}{d+1}} b^{\frac{2}{d+1}}\left((d+1)^{-\frac{d}{d+1}}+(2+\sqrt{d+1})^{\frac{2}{d+1}}\right).
	\end{equation}
	One can show that \cite{ei22learning}
	\begin{equation}\label{ineq:exercise about d}
		(d+1)^{-\frac{d}{d+1}}+(2+\sqrt{d+1})^{\frac{2}{d+1}}\le \left(e^4(d+1)\right)^{\frac{1}{d+1}},\qquad d\ge 1.
	\end{equation}
	With this inequality, we may prove $\|h-f\|_2^2\le \e$ provided that 
	$$b^2\le e^{-5}d^{-1}\e^{d+1}\left(\bh^{\le d}_{\{\pm 1\}}\right)^{-2d}.$$ 
	Plugging this in \eqref{eq:choice of N in terms of b} gives 
	\begin{equation}
		N=\frac{e^8 d^2}{\e^{d+1}}\left(\bh^{\le d}_{\{\pm 1\}}\right)^{2d}\log\left(\frac{n}{\delta}\right).
	\end{equation}
	
	To conclude, we proved the main result of \cite{ei22learning}.
	
	\begin{theorem}
		Suppose that $\epsilon,\delta\in (0,1)$, $f:\{-1,1\}^n\to [-1,1]$ is of degree at most $d$ and
		\begin{equation*}
			N\ge \frac{e^8 d^2}{\epsilon^{d+1}}\left(\textnormal{BH}^{\le d}_{\{\pm 1\}}\right)^{2d}\log\left(\frac{n}{\delta}\right).
		\end{equation*}
		Then given $N$  uniformly random independent queries 
		$$(x,f(x)),\qquad x\in \{-1,1\}^n,$$ 
		one can construct a random $h:\{-1,1\}^n\to \real $ such that 
		\begin{equation*}
			\|h-f\|_2^2\le \epsilon,
		\end{equation*}
		with probability at least $1-\delta$. 
	\end{theorem}
	
	\begin{remark}
		In view of the results of Arunachalam, Dutt, Gutiérrez and Palazuelo \cite{carlos24bh}, one may replace $\textnormal{BH}^{\le d}_{\{\pm 1\}}$ with $2^{\frac{d-1}{d}}$ when $f$ in the above theorem is Boolean.  Then the number of random queries satisfies
		\begin{equation}
			N\ge \frac{4^{d-1}e^8 d^2}{\e^{d+1}}\log\left(\frac{n}{\delta}\right).
		\end{equation}
	\end{remark}
	
	We close this section with one more application of Bohnenblust--Hille inequality to analysis of Boolean functions. A result of 
	Dinur, Friedgut, Kindler, and O’Donnell \cite{dfko07tail} states that bounded function whose Fourier coefficients decay rapidly enough are close to some \emph{juntas}. Recall that $f:\{-1,1\}^n\to \R$ is said to be a \emph{$k$-junta} if it depends on at most $k$ variables.
	This is a generalization of Bourgain's result \cite{bourgain02}  concerning Boolean functions. In the special case of bounded low-degree functions, one may give a short proof using Bohnenblust--Hille inequality.
	
	\begin{theorem}\label{thm:junta}
		Fix $d\ge 1$. Suppose that $f:\{-1,1\}^n\to [-1,1]$ is of degree at most $d$. Then for any $\e>0$, there exists a $k$-junta $g:\{-1,1\}^n\to \R$ such that 
		\begin{equation}
			\|f-g\|_2\le \e,\qquad  \textnormal{ with }\qquad k\le \frac{d \left(\bh^{\le d}_{\{\pm 1\}}\right)^{2d}}{\e^{2d}}.
		\end{equation} 
	\end{theorem}
	
	\begin{proof}
		The proof is of the same spirit as in the learning algorithm in \cite{ei22learning}. The junta $g$  is of the form 
		\begin{equation}
			g=\sum_{S:|\widehat{f}(S)|>a}\widehat{f}(S)\chi_S
		\end{equation}
		for some $a>0$. Details of the remaining arguments can be found in \cite{vz22}.
	\end{proof}
	
	\begin{remark}
		In the special case of Boolean functions of degree at most $d$, it is known that they are already $\cO(2^d)$-juntas \cite{ns94degree,hatami20degree,jake22degree}.
	\end{remark}
	
	
	\section{Bohnenblust--Hille inequality: a brief history}
	
	In 1913, Bohr \cite{bohr13} asked a problem  concerning the convergence of Dirichlet series, known as \emph{Bohr's strip problem}. The problem was answered by Bohnenblust and Hille in 1931 \cite{bh31}. In the solution, they discovered the first Bohnenblust--Hille inequality on the circle groups $\T^n$. It has two versions, one in terms of \emph{multi-linear forms}, as a generalization of the following result of Littlewood \cite{littlewood30} in 1930. Denote by $(e_j)_{1\le j\le n}$ the canonical basis of $\C^n$. We use $\D:=\{z\in \C:|z|<1\}$ to denote the open unit disc. 
	
	\begin{theorem}[Littlewood's 4/3 inequality]\label{thm:littlewood}
		There exists a constant $C>0$ such that for any $n\ge 1$ and any bilinear form $B:\C^n\times \C^n\to \C$ we have 
		\begin{equation}
			\left(\sum_{i,j=1}^{n}|B(e_i,e_j)|^{\frac{4}{3}}\right)^{\frac{3}{4}}\le C\sup_{z,w\in \D^n}|B(z,w)|.
		\end{equation}
	\end{theorem}
	
	Another version is the  \emph{polynomial Bohnenblust--Hille inequality}, as we have seen in the last section for the discrete hypercubes. The original work of Bohnenblust and Hille concerns the circle groups $\T^n$. From now on, we shall discuss the supreme norm over different sets. For this we use the notation $\|f\|_X$ for $\sup_{x\in X}|f(x)|$.  For a multi-index $\alpha=(\alpha_1,\dots, \alpha_n)$ of non-negative integers, we write  $|\alpha|:=\sum_{j=1}^{n}\alpha_j$ and 
	$$
	z^\alpha:=z_1^{\alpha_1}\cdots z_n^{\alpha_n},\qquad z=(z_1,\dots, z_n)\in \C^n.
	$$

	\begin{theorem}[Circle Bohnenblust--Hille]\label{thm:bh for torus}
		Fix $d\ge 1$. There exists a constant $C(d)>0$ such that for any $n\ge 1$ and any analytic polynomial 
		$f(z)=\sum_{|\alpha|\le d}\widehat{f}(\alpha)z^{\alpha}$ of degree at most $d$, we have 
		\begin{equation}\label{ineq:bh for torus}
			\|\widehat{f}\|_{\frac{2d}{d+1}}\le C(d)\|f\|_{\T^n}.
		\end{equation}
		Denoting the best constant by $\bh^{\le d}_{\T}$, we have $\bh^{\le d}_{\T}\le C^{\sqrt{d\log d}}$ for some universal $C>0.$
	\end{theorem}

	The sub-exponential upper bound $ C^{\sqrt{d\log d}}$ was obtained by Bayart, Pellegrino and Seoane-Sepúlveda \cite{bayart14}, improving the exponential bound in \cite{defant11bh}. The optimal bound remains open. 
	

	
	Here we will not discuss the general multi-linear Bohnenblust--Hille inequality that extends Theorem \ref{thm:littlewood}. Remark only that in the general $d$-linear form case, $d\ge 2$, the exponent $4/3$ on the left-hand side is replaced by $\frac{2d}{d+1}$.  We present a proof of Theorem \ref{thm:littlewood}  since it contains inspiring arguments for the proof of polynomial Bohnenblust--Hille inequality \eqref{ineq:bh for torus}.
	
	\begin{proof}[Proof of Theorem \ref{thm:littlewood}]
		The proof is borrowed from \cite[Chapter 6]{defant19book}. We show that
		\begin{equation}
			\left(\sum_{i,j=1}^{n}|a_{ij}|^{\frac{4}{3}}\right)^{\frac{3}{4}}\le C\sup_{z,w\in \D^n}\left |\sum_{i,j=1}^{n}a_{ij}z_i w_j\right|.
		\end{equation}
		For this, one uses twice H\"older's inequality, first for $j$ variable in view of 
		$$\frac{3}{4}=\frac{1/2}{2}+\frac{1/2}{1}$$
		and then for $i$ variable noting
		$$ \frac{3}{4}=\frac{1}{2}+\frac{1}{4}$$
		to obtain
		\begin{equation}\label{ineq:blei for two variables}
			\begin{split}
				\left(\sum_{i,j=1}^{n}|a_{ij}|^{\frac{4}{3}}\right)^{\frac{3}{4}}
				\le &\left(\sum_{i}\left(\sum_{j}|a_{ij}|^{2}\right)^{\frac{1}{4}\cdot \frac{4}{3}}\cdot \left(\sum_{j}|a_{ij}|\right)^{\frac{1}{2}\cdot \frac{4}{3}}\right)^{\frac{3}{4}}\\
				\le &\left(\sum_{i}\left(\sum_{j}|a_{ij}|^{2}\right)^{\frac{1}{4}\cdot 2}\right)^{\frac{1}{2}}\cdot \left(\sum_{i}\left(\sum_{j}|a_{ij}|\right)^{\frac{1}{2}\cdot 4}\right)^{\frac{1}{4}}\\
				=&\sqrt{\sum_{i}\left(\sum_{j}|a_{ij}|^{2}\right)^{\frac{1}{2}} \cdot \left(\sum_{i}\left(\sum_{j}|a_{ij}|\right)^{2}\right)^{\frac{1}{2}}}.
			\end{split}
		\end{equation}
		From Minkowski's inequality, we deduce
		
		\begin{align}
			\left(\sum_{i,j=1}^{n}|a_{ij}|^{\frac{4}{3}}\right)^{\frac{3}{4}}
			\le &\sqrt{\sum_{i}\left(\sum_{j}|a_{ij}|^{2}\right)^{\frac{1}{2}} \cdot \sum_{j}\left(\sum_{i}|a_{ij}|^{2}\right)^{\frac{1}{2}} }.
		\end{align}
		Now we apply Khintchin's inequality for Steinhaus random variables \cite{konig14steinhaus}
		\begin{align}
			\left(\sum_{j}|a_{ij}|^{2}\right)^{\frac{1}{2}}
			\le C\int_{\T^n}\left|\sum_{j}a_{ij}w_j\right|\d w
		\end{align}
		so that 
		\begin{equation}
			\sum_{i}\left(\sum_{j}|a_{ij}|^{2}\right)^{\frac{1}{2}}
			\le  C\sup_{w\in \D^n} \sum_{i}\left|\sum_{j}a_{ij}w_j\right|
			\le C\sup_{z,w\in \D^n} \left|\sum_{i}\sum_{j}a_{ij}z_i w_j\right|.
		\end{equation}
		Similarly, we obtain the same upper bound for $\sum_{j}\left(\sum_{i}|a_{ij}|^{2}\right)^{\frac{1}{2}}$, which concludes the proof. 
	\end{proof}
	
		
		\begin{remark}
			The exponent $4/3$ is best possible by looking at $a_{rs}=e^{\frac{2\pi i rs}{n}}$.
		\end{remark}
		
		We refer to \cite{defant19book} for more discussions of Bohnenblust--Hille inequality. 
		
		\section{Bohnenblust--Hille inequality: sketch of the proof}
		
		In this section, we give a proof sketch of the Boolean Bohnenblust--Hille \eqref{ineq:boolean bh}. For simplicity, we consider homogeneous polynomials only (in the $\T^n$ case \eqref{ineq:bh for torus} one may reduce to homogeneous polynomials with the same constants \cite{bayart14,dmp19}). That is, we want to prove
		\begin{equation}
			\|\widehat{f}\|_{\frac{2d}{d+1}}\le \bh^{=d}_{\{\pm1\}}\|f\|_{\{\pm1\}^n}
		\end{equation}
		for all $f=\sum_{|S|=d}\widehat{f}(S)\chi_S$, where we used $\bh^{=d}_{\{\pm1\}}$ to specify the best constant for the homogeneous polynomials of degree $d$. We shall employ a four-step argument to show the following inductive inequality 
		\begin{equation}\label{ineq:inductive}
			\bh^{=d}_{\{\pm1\}}
			\le C(d,k)\bh^{=k}_{\{\pm1\}},\qquad 1<k<d
		\end{equation}
		with 
		\begin{equation}
			C(d,k)=\left(\frac{k+1}{k-1}\right)^{\frac{d-k}{2}}\frac{(1+\sqrt{2})^d d^d}{k^k (d-k)^{d-k}}.
		\end{equation}
		Once we have \eqref{ineq:inductive}, we may apply it repeatedly to obtain an upper bound of $\bh^{=d}_{\{\pm1\}}$.
		
		%
		

		Before proceeding with the proof of \eqref{ineq:inductive}, we need some lemmas. The first lemma can be considered as a multi-variate generalization of \eqref{ineq:blei for two variables}. To formulate it in a compact way, we need some notation. For a multi-index $\bi=(i_1,\dots, i_d)$ and $S\subset [d]$, we use the convention 
		\begin{equation*}
			\sum_{\bi_S}=\sum_{i_j:j\in S}.
		\end{equation*}
		For $S\subset [d]$, we write $S^c=[d]\setminus S$ its complement. For example, for $S=\{1,3,5\}\subset [6]$ we have $S^c=\{2,4,6\}$ and 
		\begin{equation*}
			\sum_{\bi_S}=\sum_{i_1, i_3,i_5},\qquad 	\sum_{\bi_{S^c}}=\sum_{i_2, i_4,i_6}.
		\end{equation*}
		
		The following inequality is named after Blei and can be found in \cite{bayart14}.
		\begin{lemma}[Blei's inequality]\label{lem:blei}
			Fix $n\ge 1$ and $1\le k\le d$. For any scalar matrix $(a_{\bi})_{\bi\in [n]^d}$ we have 
			\begin{equation}\label{ineq:blei general}
				\left(\sum_{\bi\in [n]^d}|a_{\bi}|^{\frac{2d}{d+1}}\right)^{\frac{d+1}{2d}}
				\le \left[\prod_{S\subset [d], |S|=k}
				\left(\sum_{\bi_S}
				\left(\sum_{\bi_{S^c}}
				|a_{\bi}|^2\right)^{\frac{1}{2}\frac{2k}{k+1}}\right)^{\frac{k+1}{2k}}
				\right]^{\binom{d}{k}^{-1}}
			\end{equation}
		\end{lemma}
		\begin{proof}
			Similar to \eqref{ineq:blei for two variables}, the proof is again a combination of H\"older's inequality and Minkowski's inequality, but the use is more complicated. For $\ba=(a_{i_1,\dots, i_d})$ consider 
			\begin{equation}
				\|\ba\|_{\bp}:=\left(
				\sum_{i_1}\left(\cdots 
				\left(\sum_{i_{d-1}}\left(\sum_{i_d}|a_{i_1,\dots, i_d}|^{p_d}\right)^{\frac{p_{d-1}}{p_d}}\right)^{\frac{p_{d-2}}{p_{d-1}}}
				\cdots \right)^{\frac{p_1}{p_2}}
				\right)^{\frac{1}{p_1}},\qquad \bp=(p_1,\dots, p_d).
			\end{equation}
			Then one has H\"older's inequality
			\begin{equation}
				\|\ba\|_{\br}\le \|\ba\|_{\bp}^{\theta}\|\ba\|_{\bq}^{1-\theta}
			\end{equation}
			for $\bp,\bq, \br$ such that 
			\begin{equation}
				\frac{1}{r_i}=\frac{\theta}{p_i}+\frac{1-\theta}{q_i},\qquad 1\le i\le d.
			\end{equation}
			Now for any $S\subset [d]$ with $|S|=k$, we put $\bp^S=(p^S_1,\dots, p^S_d)$ such that 
			\begin{equation}
				p^S_i=
				\begin{cases}
					\frac{2k}{k+1}&i\in S\\
					2& i\notin S
				\end{cases}.
			\end{equation}
			Then for $\bp=(\frac{2d}{d+1},\cdots, \frac{2d}{d+1})$ we have 
			\begin{equation}
				\frac{1}{p_i}=\frac{d+1}{2d}=\sum_{S\subset [d], |S|=k}\frac{1/\binom{d}{k}}{p^S_i},\qquad 1\le i\le d.
			\end{equation}
			Note that the left-hand side of \eqref{ineq:blei general} is nothing but $\|\ba\|_\bp$, and we may conclude the proof using H\"older's inequality
			\begin{equation}
				\|\ba\|_{\bp}\le \prod_{S\subset [d], |S|=k}\|\ba\|^{1/\binom{d}{k}}_{\bp^S}
			\end{equation}
			and Minkowski's inequality to each $\|\ba\|_{\bp^S}$.
		\end{proof}

		The following inequality can be seen as an extension of Khintchin's inequality for Rademacher random variables. 
		\begin{lemma}[Moment comparison]\label{lem:moment comparison}
			For $f:\{-1,1\}^n\to \R$ of degree at most $d$, we have 
			\begin{equation}
				\|f\|_2\le \rho_p^{d}\|f\|_p,\qquad 1\le p<2 
			\end{equation}
			with $\rho_p=(p-1)^{-1/2}$ for $1<p<2$ and $\rho_1=e.$
		\end{lemma}
		
		\begin{proof}
			This is a standard consequence of hypercontractivity: for $1<p<q<\infty$
			\begin{equation}
				\|P_t f\|_q\le \|f\|_p, \qquad t\ge \frac{1}{2}\log\frac{q-1}{p-1}
			\end{equation}
			where $P_t$ is the heat semigroup on discrete hypercubes given by 
			$$P_t\left( 
			\sum_{S\subset [n]}\widehat{f}(S)\chi_S
			\right)=\sum_{S\subset [n]}e^{-t|S|}\widehat{f}(S)\chi_S.
			$$
			See \cite{odonnell14book} for more discussions. 
		\end{proof}
		
		The following \emph{polarization} result suffices the proof of Boolean Bohnenblust--Hille inequality for homogeneous polynomials under consideration. For the proof for general low-degree polynomials, one needs a variant that can be found in \cite{dmp19}.  Recall that to any homogeneous polynomial  $P:\R^n\to \R$ of degree $d$ 
		\begin{equation}
			P(x_1,\dots, x_n)=\sum_{1\le i_1<\cdots<i_d\le n} a_{i_1,\dots, i_d} x_{i_1}\cdots x_{i_d}
		\end{equation}
		there associates a unique $d$-linear symmetric form $L=L_P:(\R^n)^d\to \R$ such that 
		\begin{equation}
			L_P(x,\dots, x)=P(x),\qquad x\in \R^n.
		\end{equation}
		In fact, $L_P$ has an explicit form
		\begin{equation}
			L_P(x^{(1)},\dots, x^{(d)})=\sum_{j_1,\dots, j_d} c_{j_1,\dots, j_d}x^{(1)}_{j_1}\cdots x^{(d)}_{j_d},
		\end{equation}
		where 
		\begin{equation}
			c_{j_1,\dots, j_d}=
			\begin{cases}
				\frac{a_{i_1,\dots, i_d}}{d!} & \textnormal{if } \{j_1,\dots, j_d\}=\{i_1,\dots, i_d\}\\
				0&\textnormal{otherwise}
			\end{cases}.
		\end{equation}

		\begin{lemma}[Polarization]\label{lem:polarization}
			Let $P:\R^n\to \R$ be a homogeneous polynomial of degree $d$, and let $L=L_P:(\R^n)^d\to \R$ be associated $d$-linear symmetric form. Then for all $1\le k\le d$ and all $x,y\in [-1,1]^n$
			\begin{equation}\label{ineq:polarization}
				|L(\underbrace{x,\dots,x}_{k},\underbrace{y,\dots, y}_{d-k})|
				\le \frac{(1+\sqrt{2})^d d^d}{k^k (d-k)^{d-k}}\frac{k!(d-k)!}{d!}\|P\|_{[-1,1]^n}.
			\end{equation}
		\end{lemma}
		
		\begin{proof}
			The proof follows from a classical inequality of Markov \cite[page 248]{be95polynomial}: For a real polynomial $p(t)=\sum_{m=0}^{d}a_m t^m$ of degree at most $d$ one has 
			\begin{equation}
				|a_m|\le M_{m,d}\|p\|_{[-1,1]},\qquad 0\le m\le d.
			\end{equation}
			Here $M_{,m,d}$ is the Markov's number with explicit expression and has the estimate $M_{m,d}\le (1+\sqrt{2})^d$ \cite{dmp19}. Then the proof is finished by considering 
			$$
			p(t):=L(\lambda tx+y,\lambda tx+y,\dots, \lambda tx+y)=P(\lambda tx+y),
			$$
			with $\lambda>0$ to be chosen later. In fact, writing $p(t)=\sum_{m=0}^{d}a_m t^m$, the left-hand side of \eqref{ineq:polarization} is exactly $\lambda^{-k}\binom{d}{k}^{-1}|a_k|$, while Markov's inequality gives 
			$$
			|a_k|\le (1+\sqrt{2})^d \|p\|_{[-1,1]}\le  (1+\sqrt{2})^d \|P\|_{[-(\lambda +1),\lambda +1]^n}\le (\lambda+1)^d(1+\sqrt{2})^d\|P\|_{[-1,1]^n}.
			$$
			Then the proof is finished by choosing $\lambda=\frac{k}{d-k}$.
		\end{proof}
		
		For $d,n\ge 1$, we consider the subset $\cj(d,n)$ of $[n]^d$ given by 
		$$\cj(d,n):=\left\{\bj \in [n]^d: 1\le j_1<\dots <j_d\le n\right\}.$$
		For any $\bi\in [n]^d$ and $x\in \{-1,1\}^n$ we write 
		\begin{equation}
			x_{\bi}=\prod_{k=1}^{d}x_{\bi_k}.
		\end{equation}
		Then any $f=\sum_{|S|=d}\widehat{f}(S)\chi_S$ can be written as 
		\begin{equation}
			f(x)=\sum_{\bj\in\cj(d,n)}a_{\bj}x_{\bj}
		\end{equation}
		with $\|\widehat{f}\|_p=\|(a_{\bj})_{\bj}\|_p$.  Now we are ready to present the four-step argument to prove 
		\begin{equation}\label{ineq:homo bh}
			\|\widehat{f}\|_{\frac{2d}{d+1}}\le \bh^{=d}_{\{\pm1\}}\|f\|_{\{\pm1\}^n},\qquad f=\sum_{|S|=d}\widehat{f}(S)\chi_S.
		\end{equation}
		
		\textbf{Step 1} By Lemma \ref{lem:blei}
		\begin{equation}
			\|\widehat{f}\|_{\frac{2d}{d+1}}=\left(\sum_{\bj\in \cj(d,n)}|a_\bj|^{\frac{2d}{d+1}}\right)^{\frac{d+1}{2d}}
			\le \left[\prod_{S\subset [d], |S|=k} 
			\left(\sum_{\bj_S}
			\left(\sum_{\bj_{S^c}}
			|a_{\bj}|^2
			\right)^{\frac{1}{2}\frac{2k}{k+1}}
			\right)^{\frac{k+1}{2k}}
			\right]^{\binom{d}{k}^{-1}}.
		\end{equation}
		
		\textbf{Step 2} We fix $S\subset [d]$ with $|S|=k$. By Parseval's identity and Lemma \ref{lem:moment comparison}
		\begin{equation}
			\left(\sum_{\bj_{S^c}}
			|a_{\bj}|^2
			\right)^{\frac{1}{2}\frac{2k}{k+1}}
			\le \rho_{\frac{2k}{k+1}}^{d-k}\E_{y}\left| \sum_{\bj_{S^c}} a_{\bj} y_{\bj_{S^c}}\right|^{\frac{2k}{k+1}}.
		\end{equation}
		Thus
		\begin{align}
			\sum_{\bj_S}
			\left(\sum_{\bj_{S^c}}
			|a_{\bj}|^2
			\right)^{\frac{1}{2}\frac{2k}{k+1}}
			\le \rho_{\frac{2k}{k+1}}^{d-k}\E_{y} \sum_{\bj_S}\left| \sum_{\bj_{S^c}} a_{\bj} y_{\bj_{S^c}}\right|^{\frac{2k}{k+1}}
			\le \rho_{\frac{2k}{k+1}}^{d-k}\sup_{y\in \{-1,1\}^n} \sum_{\bj_S}\left| \sum_{\bj_{S^c}} a_{\bj} y_{\bj_{S^c}}\right|^{\frac{2k}{k+1}}.
		\end{align}
		
		\textbf{Step 3} Applying the Boolean Bohnenblust--Hille inequality \eqref{ineq:homo bh} for homogeneous polynomials of degree $k$, we obtain
		\begin{equation}
			\sum_{\bj_S}\left| \sum_{\bj_{S^c}} a_{\bj} y_{\bj_{S^c}}\right|^{\frac{2k}{k+1}}
			\le \left(\bh^{=k}_{\{\pm1\}}\right)^{\frac{2k}{k+1}}\sup_{x\in \{-1,1\}^n} \left|\sum_{\bj_S} \sum_{\bj_{S^c}} a_{\bj} x_{\bj_S} y_{\bj_{S^c}}\right|^{\frac{2k}{k+1}}.
		\end{equation}
		Therefore, 
		\begin{align}
			\left(\sum_{\bj_S}
			\left(\sum_{\bj_{S^c}}
			|a_{\bj}|^2
			\right)^{\frac{1}{2}\frac{2k}{k+1}}
			\right)^{\frac{k+1}{2k}}
			&\le \rho_{\frac{2k}{k+1}}^{d-k} \bh^{=k}_{\{\pm1\}}\sup_{x,y\in \{-1,1\}^n} \left|\sum_{\bj_S} \sum_{\bj_{S^c}} a_{\bj} x_{\bj_S} y_{\bj_{S^c}}\right|.
		\end{align}
		
		\textbf{Step 4} Let us extend the definition of $a_{\bj}$ to $\bi\in[n]^d$:
		\begin{equation}
			a_{\bi}=
			\begin{cases}
				a_{\bj}& \textnormal{ if } \{\bi_1,\dots, \bi_d\}=\{\bj_1<\dots<\bj_d\} \textnormal{ for some }\bj\in \cj(d,n)\\
				0 & \textnormal{ otherwise }
			\end{cases}.
		\end{equation}
		So 
		\begin{equation}
			\sum_{\bj_S} \sum_{\bj_{S^c}} a_{\bj} x_{\bj_S} y_{\bj_{S^c}}
			=	\frac{1}{k!(d-k)!}\sum_{\bi_S} \sum_{\bi_{S^c}} a_{\bi} x_{\bi_S} y_{\bi_{S^c}}
			=	\frac{d!}{k!(d-k)!}L_f(\underbrace{x,\dots,x}_{k},\underbrace{y,\dots, y}_{d-k})
		\end{equation}
		where $L_f$ is the $d$-linear symmetric form associated to $f$. By polarization Lemma \ref{lem:polarization}, we have
		\begin{equation}
			|L_f(\underbrace{x,\dots,x}_{k},\underbrace{y,\dots, y}_{d-k})|
			\le \frac{(1+\sqrt{2})^d d^d}{k^k (d-k)^{d-k}}\frac{k!(d-k)!}{d!}\|f\|_{\{\pm 1 \}^n}.
		\end{equation}
		All combined, we showed that 
		\begin{equation}
			\left(\sum_{\bj\in \cj(d,n)}|a_\bj|^{\frac{2d}{d+1}}\right)^{\frac{d+1}{2d}}
			\le \rho_{\frac{2k}{k+1}}^{d-k} \bh^{=k}_{\{\pm1\}}\frac{(1+\sqrt{2})^d d^d}{k^k (d-k)^{d-k}}\|f\|_{\{\pm 1 \}^n}
		\end{equation}
		which entails the desired inductive inequality
		\begin{equation}
			\bh^{=d}_{\{\pm1\}}
			\le =\left(\frac{k+1}{k-1}\right)^{\frac{d-k}{2}}\frac{(1+\sqrt{2})^d d^d}{k^k (d-k)^{d-k}}\bh^{=k}_{\{\pm1\}}
		\end{equation}
		in view of $\rho_p^2=\frac{1}{p-1}$ when $1<p\le 2.$

		\chapter{Bohnenblust--Hille inequality on quantum systems via a reduction method} 
		
		\section{Qubit Bohnenblust--Hille inequality}
		
		
		We use $M_2(\C)$ to denote the 2-by-2 complex matrix algebra, and $M_2(\C)^{\otimes n}\simeq M_{2^n}(\C)$ its $n$-fold tensor product. The Pauli matrices
		
		\begin{equation*}
			\sigma_0=\begin{pmatrix}1&0\\0&1\end{pmatrix},\quad \sigma_1=\begin{pmatrix}0&1\\1&0\end{pmatrix},\quad \sigma_2=\begin{pmatrix}0&-i\\i&0\end{pmatrix},\quad \sigma_3=\begin{pmatrix}1&0\\0&-1\end{pmatrix},
		\end{equation*}
		form a basis of $M_2(\C)$. For $\bs=(s_1,\dots, s_n)\in\{0,1,2,3\}^n$, we put
		\begin{equation*}
			\sigma_\bs:=\sigma_{s_1}\otimes\dots\otimes \sigma_{s_n}\, .
		\end{equation*}
		All the $\sigma_{\bs}, \bs\in \{0,1,2,3\}^n$ form a basis of $M_2(\C)^{\otimes n}$ and play the role of characters $\chi_S,S\subset [n]$ on $\{-1,1\}^n$. Any $A\in M_2(\C)^{\otimes n}$ has the unique Fourier expansion
		\begin{equation*}
			A=\sum_{\bs\in\{0,1,2,3\}^n} \widehat{A}_\bs \,\sigma_\bs
		\end{equation*}
		with  $\widehat{A}_\bs\in \C$ being the Fourier coefficient. For any  $\bs=(s_1,\dots, s_n)\in\{0,1,2,3\}^n$, we denote by $|\bs|$ the number of non-zero $s_j$'s. Similar to the classical setting, $A\in M_2(\C)^{\otimes n}$ is {\em of degree at most $d$}   if $\widehat{A}_\bs =0$ whenever $|\bs|>d$, and it is {\em homogeneous of degree $d$} if $\widehat{A}_\bs=0$ whenever $|\bs|\neq d$. 
		
		Recall that $\sigma_j, 1\le j\le 3$ satisfy the anti-commutation relation:
		\begin{equation}\label{eq:pauli commutation}
			\sigma_j\sigma_k+\sigma_k\sigma_j=2\delta_{jk}\un,\qquad 1\le j,k\le 3
		\end{equation}
		which will play a key role in the following. Here and in what follows, $\un$ denotes the identity matrix. 
		
		
		Learning quantum observables $A\in M_2(\C)^{\otimes n}$ has been quite popular in recent years. We are not going to survey the progress this direction in any sense. But the Fourier analysis tools in the qubit systems can be as useful as they are in the classical case (Boolean cubes). Of particular interest is the following Bohnenblust--Hille inequality for the qubit system which is a natural question. Here and in what follows, we use $\|A\|$ to denote the operator norm of a matrix $A$.
		
		\begin{theorem}[Qubit Bohnenblust--Hille]\label{thm:qubit bh}
			Fix $d\ge 1$. There exists $C(d)>0$ such that for all $n\ge 1$ and all $A\in M_2(\C)^{\otimes n}$ of degree at most $d$, we have 
			\begin{equation}\label{ineq:qubit bh}
				\|\widehat{A}\|_{\frac{2d}{d+1}}\le C(d)\|A\|.
			\end{equation}
		\end{theorem}
		
		This inequality \eqref{ineq:qubit bh} was conjectured by Rouz\'e, Wirth and Zhang \cite{rwz24kkl}, and was first proved by Huang, Chen and Preskill \cite{hcp23learning} in their study of learning arbitrary quantum progress. Their proof essentially follows the scheme presented in the last chapter and they achieved a constant of the order $d^{\cO(d)}$. Later on, Volberg and Zhang \cite{vz22} gave another proof that reduces the problem to the classical case, and they obtained a constant of exponential growth. This reduction method will be the main focus of this lecture and we shall explain it in the next section. 
		
		We will not discuss applications of \eqref{ineq:qubit bh}, such as learning quantum observables, which can be analogous to the classical case. But let us remark that in the work of Huang--Chen--Preskill, they actually try to learn arbitrary quantum observables via  a low-degree truncation, and then treat the low-degree and tail parts separately. 
		
		\section{A reduction method in the qubit system}
		
		The idea of reduction method in \cite{vz22} is very simple. To any  $A\in M_2(\C)^{\otimes n}$ of degree at most $d$, we aim to find a classical function $f_A:\{-1,1\}^m\to \C$ of the same degree such that 
		\begin{equation}
			\|\widehat{A}\|_{\frac{2d}{d+1}}\precsim \|\widehat{f_A}\|_{\frac{2d}{d+1}}\precsim \|f_A\|_{\infty}\precsim\|A\|,
		\end{equation}
		with dimension-free constants. 
		The inequality in the middle comes for ``free" since it is nothing but the classical Boolean Bohnenblust--Hille inequality \eqref{ineq:boolean bh}. The difficulty is to realize the first and last inequalities simultaneously with dimension-free constants. Note that $m$ does not have to be $n$ and actually we will choose $m=3n$.
		
		In the sequel, we use $\langle \cdot, \cdot\rangle$ to denote the inner product on $\C^n$ that is linear in the second argument. For a linear operator $A$, we use $A^\dagger$ to denote its adjoint with respect to $\langle \cdot, \cdot\rangle$. For a unit vector $\eta$ we shall use the convention that $\ket{\eta}\bra{\eta}$ denotes the corresponding rank-one projection. 
		
		We start with a simple fact. 
		
		\begin{lemma}\label{lem:qubit orthogonal}
			Suppose that two Hermitian matrices $A,B$ such that $AB=-BA$. If $\eta\neq \vec{0}$ is an eigenvector of $B$ with eigenvalue $\lambda\neq 0$. Then $\langle \eta, A\eta\rangle =0$.
		\end{lemma}
		
		\begin{proof}
			By definition, 
			\begin{equation}
				\lambda\langle \eta, A\eta\rangle 
				=\langle \eta, AB\eta\rangle
				=-\langle \eta, BA\eta\rangle
				=-\langle B\eta, A\eta\rangle
				=-\lambda\langle \eta, A\eta\rangle.
			\end{equation}
			So $2\lambda\langle \eta, A\eta\rangle=0$. Since $\lambda \neq 0$, we have $\langle \eta, A\eta\rangle =0$ as desired. 
		\end{proof}
		
		Now we prove Theorem \ref{thm:qubit bh} following the approach in  \cite{vz22}.
		
		\begin{proof}[Proof of Theorem \ref{thm:qubit bh}]
			Note first that each of $\sigma_j,1\le j\le 3$ has $\pm 1$ as eigenvalues. For any $\kappa\in \{1,2,3\}$ and $\e\in\{-1,1\}$, we denote by $e^\kappa_{\e}$ the unit eigenvector of $\sigma_\kappa$ corresponding to $\e$.
			For  any 
			$$\vec{\e}:=\left(\e^{(1)}_1,\dots, \e^{(1)}_n,\e^{(2)}_1,\dots, \e^{(2)}_n,\e^{(3)}_1,\dots, \e^{(3)}_n\right)\in \{-1,1\}^{3n},$$ 
			consider the matrix $\rho(\vec{\e})$ that is defined as follows
			\begin{equation*}
				\rho(\vec{\e}):=\rho_1(\vec{\e})\otimes \cdots \otimes \rho_n (\vec{\e})\in M_2(\C)^{\otimes n}, 
			\end{equation*}
			where for each $1\le j\le n$
			\begin{equation*}
				\rho_j(\vec{\e}):=\frac{1}{3}\ket{e^1_{\e^{(1)}_j}}\bra{e^1_{\e^{(1)}_j}}+\frac{1}{3}\ket{e^{2}_{\e^{(2)}_j}}\bra{e^2_{\e^{(2)}_j}}+\frac{1}{3}\ket{e^3_{\e^{(3)}_j}}\bra{e^3_{\e^{(3)}_j}}.
			\end{equation*}
			By definition,  each $\rho_j(\vec{\e})$ is a density matrix, that is, positive semi-definite with trace $1$. So $\rho(\vec{\e})$ is also a density matrix. For any $A\in M_2(\C)^{\otimes n}$, we define $f_A:\{-1,1\}^{3n}\to \C$ as
			\begin{equation}
				f_A(\vec{\e}):=\tr[A \rho(\vec{\e})].
			\end{equation}		
			Since $\rho(\vec{\e})$ is a density matrix, we have by duality
			\begin{equation}
				\|f_A\|_{\{\pm 1\}^{3n}}\le \|A\|.
			\end{equation}
			Now let us look at the Fourier expansion of $f_A$. For this we rewrite the Fourier expansion of $A$ in a different way. Consider the operator 
			\begin{equation}
				\sigma^{\kappa_1,\dots,  \kappa_l}_{i_1,\dots, i_l}:=\cdots\otimes  \sigma_{\kappa_1}\otimes \cdots \otimes \sigma_{\kappa_l}\otimes \cdots,
			\end{equation}
			where $\sigma_{\kappa_j}$ appears in the $i_j$-th place for each $1\le j\le l$, and all other $(n-l)$ components are simply the  identity matrices $\sigma_0=\un$. Then for $A\in M_2(\C)^{\otimes n}$ of degree at most $d$, its Fourier expansion also takes the form
			\begin{equation}
				A=\sum_{0\le l\le d}\sum_{\kappa_1,\dots, \kappa_l\in \{1,2,3\}}\sum_{1\le i_1<\cdots <i_l\le n}a^{\kappa_1,\dots,  \kappa_l}_{i_1,\dots, i_l}\sigma^{\kappa_1,\dots,  \kappa_l}_{i_1,\dots, i_l}.
			\end{equation}
			From Lemma \ref{lem:qubit orthogonal} and \eqref{eq:pauli commutation} one deduces
			\begin{equation}
				\tr[\sigma_i \rho_j(\vec{\e})]=\frac{1}{3}\e^{(i)}_{j},\qquad 1\le i\le 3,\quad 1\le j\le n.
			\end{equation}
			Since $\rho_j(\vec{\e})$ is a density matrix, we have $\tr[\sigma_0 \rho_j(\vec{\e})]=1$. Therefore
			\begin{equation}
				\tr[\sigma^{\kappa_1,\dots,  \kappa_l}_{i_1,\dots, i_l}\rho(\vec{\e})]
				=\tr[\sigma_{\kappa_1}\rho_{i_1}]\cdots \tr[\sigma_{\kappa_l}\rho_{i_l}]
				=\frac{1}{3^l}\epsilon^{(\kappa_1)}_{i_1}\cdots\epsilon^{(\kappa_l)}_{i_l}
			\end{equation}
			and 
			\begin{equation}
				f_A(\vec{\e})
				=\sum_{0\le l\le d}\sum_{\kappa_1,\dots, \kappa_l\in \{1,2,3\}}\sum_{1\le i_1<\cdots <i_l\le n}3^{-l}a^{\kappa_1,\dots,  \kappa_l}_{i_1,\dots, i_l}\epsilon^{(\kappa_1)}_{i_1}\cdots\epsilon^{(\kappa_l)}_{i_l}
			\end{equation}
			which is exactly the Fourier expansion of $f_A$ since the multi-linear monomials $\epsilon^{(\kappa_1)}_{i_1}\cdots\epsilon^{(\kappa_l)}_{i_l}$ differ for distinct $(l;\kappa_1,\dots, \kappa_l;i_1,\dots, i_l)$'s. From this we deduce that for all $p>0$
			\begin{equation}
				\|\widehat{f_A}\|_p
				=\left(\sum_{0\le l\le d}\sum_{\kappa_1,\dots, \kappa_l}\sum_{ i_1<\cdots <i_l} 3^{-pl}|a^{\kappa_1,\dots,  \kappa_l}_{i_1,\dots, i_l}|^{p}\right)^{1/p}
				\ge 3^{-d}\left(\sum_{0\le l\le d}\sum_{\kappa_1,\dots, \kappa_l}\sum_{ i_1<\cdots <i_l} |a^{\kappa_1,\dots,  \kappa_l}_{i_1,\dots, i_l}|^{p}\right)^{1/p}
				= 3^{-d}\|\widehat{A}\|_p.
			\end{equation}
			It is clear that $\deg (f_A)=\deg(A)$, so by Boolean Bohnenblust--Hille inequality \eqref{ineq:boolean bh}
			\begin{equation}
				\|\widehat{f_A}\|_{\frac{2d}{d+1}}\le \bh^{\le d}_{\{\pm1\}}\|f_A\|_{\{\pm1\}^{3n}}.
			\end{equation}
			All combined, we obtain
			\begin{equation}
				\|\widehat{A}\|_{\frac{2d}{d+1}}
				\le 3^d \|\widehat{f_A}\|_{\frac{2d}{d+1}}
				\le 3^d \bh^{\le d}_{\{\pm1\}}\|f_A\|_{\{\pm1\}^{3n}}
				\le 3^d \bh^{\le d}_{\{\pm1\}}\|A\|,
			\end{equation}
			which concludes the proof. 
		\end{proof}

		This reduction method also works for many other Fourier analysis problems in the qubit systems \cite{vz22}.

		\section{Qudit Bohnenblust--Hille inequality and the reduction}
		
		
		How about the \emph{qudit} systems? In this context, 2-by-2 matrices are replaced by $\go$-by-$\go$ matrices for $\go \ge 3$. Certain generalizations of Pauli matrices are needed to form a basis of the $\go\times \go$ complex matrix algebra $M_{\go}(\C)$. Here we consider the \emph{Heisenberg--Weyl basis} and we refer to \cite{svz24qudit} and references therein for other possible proposals. 
		
		Fix $\go\geq 3$ and we use $\Z_\go=\{0,1,\dots, \go-1\}$ to denote the additive group of order $\go$. In the following, the operations in $\Z_\go$ are always understood as $\mod \go$ unless otherwise stated. For example, $(\ket{j})_{j\in\Z_\go}$ is the canonical basis of $\C^\go$ and $\ket{\go+j}=\ket{j}$. Let $\omega =\omega_\go= \exp(2\pi i/\go)$ and denote by $\Omega_\go=\{1,\omega,\dots, \omega^{\go-1}\}$ the multiplicative group of order $\go$. 
		
		\begin{definition}
			Define the $\go$-dimensional \emph{clock} and \emph{shift} matrices respectively via
			\begin{equation}
				Z\ket{j} = \om^j \ket{j},\qquad 
				X \ket{j}= \ket{j+1}\qquad \textnormal{ for all} \qquad j\in \mathbb{Z}_\go.
			\end{equation}
			Note that $X^\go=Z^\go=\un$. 
			Then the Heisenberg--Weyl basis for $M_\go(\C)$ is
			\[
			\textnormal{HW}(\go):=\{X^\ell Z^m\}_{\ell,m\in \Z_\go}\,.
			\]
		\end{definition}
		It is a simple exercise to verify that $\{X^\ell Z^m\}_{\ell,m\in \Z_\go}$ form a basis of $M_\go(\C)$.

		Any observable $A\in M_\go(\C)^{\otimes n}$ has a unique Fourier expansion with respect to $\textnormal{HW}(\go)$:
		\begin{equation}
			\label{expHW}
			A=\sum_{\vec{\ell},\vec{m}\in \mathbb{Z}_\go^n}\widehat{A}(\vec{\ell},\vec{m})X^{\ell_1}Z^{m_1}\otimes \cdots \otimes X^{\ell_n}Z^{m_n},
		\end{equation}
		where $\widehat{A}(\vec{\ell},\vec{m})\in\C$ is the Fourier coefficient at $(\vec{\ell},\vec{m})$ with 
		$$
		\vec{\ell}=(\ell_1,\dots, \ell_n),\qquad \vec{m}=(m_1,\dots, m_n).
		$$
		
		We say that $A$ is \emph{of degree at most $d$} if $\widehat{A}(\vec{\ell},\vec{m})=0$ whenever 
		\begin{equation*}
			|(\vec{\ell},\vec{m})|:=\sum_{j=1}^{n}(\ell_j+m_j)>d.
		\end{equation*}
		Here, $0\le \ell_j,m_j\le \go-1$.

		Here are some facts about the Heisenberg--Weyl basis. We use the convention that for $g$ an element in a group $G$, $\langle g\rangle$ denotes the abelian subgroup generated by $g$. Recall that $\gcd(a,b)$ denotes the greatest common divisor of two positive integers $a$ and $b$. 
		
		\begin{lemma}\label{lem:HW basis general}
			Under the above notations, we have the following properties. 
			\begin{enumerate}
				\item For all $k,\ell,m\in \mathbb{Z}_\go$:
				\begin{equation*}
					(X^\ell Z^m)^k=\om^{\frac{1}{2}k(k-1)\ell m}X^{k\ell} Z^{km}
				\end{equation*}
				and for all $\ell_1,\ell_2,m_1,m_2\in\mathbb{Z}_\go$:
				\begin{equation*}
					X^{\ell_1}Z^{m_1}\cdot X^{\ell_2}Z^{m_2}=\om^{\ell_2 m_1-\ell_1 m_2} X^{\ell_2}Z^{m_2}\cdot X^{\ell_1}Z^{m_1}.
				\end{equation*}
				\item If $\ell_1,m_1\in \{1,2,\dots, \go\}$ are such that $\gcd(\ell_1,m_1)=1$, and $(\ell,m)\notin \langle(\ell_1,m_1)\rangle $, then
				\begin{equation}
					X^{\ell_1}Z^{m_1}\cdot X^{\ell}Z^{m}
					= \omega^{\ell m_1-\ell_1 m}X^{\ell}Z^{m}\cdot X^{\ell_1}Z^{m_1}
				\end{equation}
				with $\omega^{\ell m_1-\ell_1 m}\neq 1$.
				
				\item If $\gcd(\ell,m)=1$, then the set of eigenvalues of $X^\ell Z^m$ is either $\Omega_\go$ or $\Omega_{2\go}\setminus \Omega_{\go}$.
			\end{enumerate}
		\end{lemma}
		
		\begin{proof}
			The proof is left as exercises, and one can find details in \cite{svz24qudit}. We only highlight the following key property in the proof 
			\begin{equation}
				ZX=\omega XZ
			\end{equation}
			which follows immediately from the definition. 
		\end{proof}
		It is also an interesting exercise to find all the eigenvectors of $X^\ell Z^m$ for general $\go\ge 3$.
		
		\medskip
		
		Now let us come back to our question: Do we have a  Bohnenblust--Hille inequality for the qudit system described above? The answer is affirmative. Adapting the proof of qubit case ($\go=2$) to the general qudit case ($\go \ge 3$), one sees the connection to the cyclic group of order $\go$ in view of the eigenvalue information of certain $X^\ell Z^m$ by Lemma \ref{lem:HW basis general}. As we shall see below, we can reduce the qudit Bohnenblust--Hille inequality to its analogs on classical cyclic groups. The bad news was that we did not know the validness of Bohnenblust--Hille inequality on cyclic groups $\Omega_\go$ for $\go \ge 3$ until recently. This result in the classical world will be the topic of next lecture, but let us state the results here for later use. 
		
		\begin{theorem}[Cyclic Bohnenblust--Hille inequality]\label{thm:bh cyclic qudit}
			Fix $\go \ge 3$ and $d\ge 1$. Then there exists a constant $C(d,\go)>0$ such that for any $n\ge 1$ and any $f:\Omega_\go^n\to \C$ of degree at most $d$, we have 
			\begin{equation}
				\|\widehat{f}\|_{\frac{2d}{d+1}}\le C(d,\go)\|f\|_{\Omega_\go^n}.
			\end{equation}
			Denoting the best constant by $\bh^{\le d}_{\Omega_\go}$, we have $\bh^{\le d}_{\Omega_\go}\le C_\go^{d^2}$ when $\go$ is prime.
		\end{theorem}
		
		We refer to the next chapter for more details of the above theorem. In the next, we shall reduce the qudit Bohnenblust--Hille inequality to this result. To adapt the reduction method to the qudit case, we need to construct suitable density matrices which come from eigen-projections of certain $X^\ell Z^m$'S. On one hand, we need to select enough number of $X^\ell Z^m $ to remember all the (non-zero) Fourier coefficients of $A$. On the one hand, choosing too many $X^\ell Z^m$ may make the commutation relations messy since their commutation relations are more complicated in view of Lemma \ref{lem:HW basis general}.
		
		In the following, we only present the qudit Bohnenblust--Hille inequality and its reduction for prime $\go\ge 3$. When $\go\ge 4$ is non-prime, the result looks a bit different and the proof is more involved \cite{svz24qudit}.
		
		\begin{theorem}[Qudit Bohnenblust--Hille, Heisenberg--Weyl Basis: prime case]\label{thm:bh HW prime}
			Fix a prime number $\go\ge 3$ and suppose $d\ge 1$.
			Consider an observable $A\in M_\go(\C)^{\otimes n}$ of degree at most $d$.
			Then we have
			\begin{equation}
				\label{ineq:bh-hw prime}
				\|\widehat{A}\|_{\frac{2d}{d+1}}\le C(d,\go)\|A\|,
			\end{equation}
			with $C(d,\go)\le (\go+1)^d \textnormal{BH}_{\Omega_\go}^{\le d}$.
		\end{theorem}
		
		Let us record the following observation as a lemma, as an extension of Lemma \ref{lem:qubit orthogonal}.
		
		\begin{lemma}\label{lem:orthogonal}
			Suppose that $k\ge 1$, $A,B$ are two unitary matrices such that $B^k=\un$, $AB=\lambda BA$ with $\lambda\in\C$ and $\lambda\ne 1$. 
			If $\eta\neq \vec{0}$ is an eigenvector of $B$ with eigenvalue $\mu$ ($\mu\ne 0$ since $\mu^k=1$), then 
			\begin{equation*}
				\langle \eta,A\eta\rangle=0.
			\end{equation*}
		\end{lemma}
		
		\begin{proof}
			By assumption
			\[
			\mu\langle \eta,A\eta\rangle
			=\langle \eta,AB\eta\rangle 
			=\lambda \langle \eta,BA\eta\rangle.
			\]
			Since $B^\dagger=B^{k-1}$, $B^\dagger\eta=B^{k-1}\eta=\mu^{k-1}\eta=\overline{\mu}\eta$. Thus
			\[
			\mu\langle \eta,A\eta\rangle
			=\lambda \langle \eta,BA\eta\rangle
			=\lambda \langle B^\dagger\eta,A\eta\rangle
			=\lambda \mu\langle \eta,A\eta\rangle.
			\]
			Hence, $\mu(\lambda-1)\langle \eta,A\eta\rangle=0$. This gives $\langle \eta,A\eta\rangle=0$ as $\mu(\lambda-1)\ne 0$. 
		\end{proof}
		
		When $\go$ is prime, the basis $\{X^\ell Z^m\}$ has nicer properties. 
		
		\begin{lemma}\label{lem:HW basis prime}
			Fix $\go\ge 3$ a prime number. Consider the set
			\begin{equation}\label{eq:Sigma prime}
				\Sigma_{\go}:=\{(1,0),(1,1),\dots, (1,\go-1),(0,1)\}.
			\end{equation}
			Then the group $\Z_{\go}\times \Z_{\go}$ is the union of subgroups 
			\begin{equation}
				\Z_{\go}\times \Z_{\go}=\bigcup_{(\ell,m)\in \Sigma_{\go}}\langle (\ell,m)\rangle
			\end{equation}
			where each two subgroups intersects with the unit $(0,0)$ only. Moreover, for any $(\ell,m)\in \Sigma_\go$, the set of eigenvalues of each $X^\ell Z^m$ is exactly $\Omega_\go.$
		\end{lemma}
		
		\begin{proof}
			%
			The proof is left as an exercise. See \cite{svz24qudit} for details.
		\end{proof}
		

		Now we are ready to prove Theorem \ref{thm:bh HW prime}.
		\begin{proof}[Proof of Theorem \ref{thm:bh HW prime}]
			Fix a prime number $\go\ge 3$. Recall that $\om=e^{\frac{2\pi i}{\go}}$. Consider $\Sigma_\go$ defined in \eqref{eq:Sigma prime}.
			For any $(\ell,m)\in \Sigma_\go$, by Lemma \ref{lem:HW basis prime} any $z\in\Om_\go$ is an eigenvalue of $X^\ell Z^m$ and we denote by $e^{\ell,m}_{z}$ the corresponding unit eigenvector.
			For any vector $\vec{\om}\in \Om_\go^{(\go+1)n}$ of the form (noting that $|\Sigma_\go|=\go+1$)
			\begin{equation}\label{eq:defn of vec omega}
				\vec{\om}=({\vec{\om}}^{\ell,m})_{(\ell,m)\in \Sigma_\go},
				\qquad {\vec{\om}}^{\ell,m}=(\omega^{\ell,m}_1,\dots, \omega^{\ell,m}_n)\in \Om_\go^{n},
			\end{equation}
			we consider the matrix 
			\[
			\rho(\vec{\om}):=\rho_{1}(\vec{\om})\otimes \cdots\otimes \rho_{n}(\vec{\om})
			\]
			where 
			\[
			\rho_k(\vec{\om}):=\frac{1}{\go+1}\sum_{(\ell,m)\in \Sigma_\go}\ketbra{e^{\ell,m}_{\om^{\ell,m}_k}}{e^{\ell,m}_{\om^{\ell,m}_k}}.
			\]
			Then each $\rho_{k}(\vec{\om})$ is a density matrix and so is $\rho(\vec{\om})$. 
			
			Fix $(\ell,m)\in \Sigma_\go$ and $1\le k\le \go-1$. We have by Lemma \ref{lem:HW basis general}
			\begin{align*}
				\tr[X^{k\ell}Z^{k m}\ketbra{e^{\ell,m}_{z}}{e^{\ell,m}_{z}}]
				&=\om^{-\frac{1}{2}k(k-1)\ell m}\langle  e^{\ell,m}_{z},(X^\ell Z^m)^k e^{\ell,m}_{z}\rangle\\
				&=\om^{-\frac{1}{2}k(k-1)\ell m}z^k, \qquad z\in \Om_\go.
			\end{align*}
			On the other hand, for any $(\ell,m)\neq (\ell',m')\in \Sigma_\go$, we have $(k\ell,km)\notin\langle (\ell',m')\rangle$ by Lemma \ref{lem:HW basis prime}.
			From our choice $\gcd (\ell',m')=1$. So Lemma \ref{lem:HW basis general} gives 
			\begin{equation*}
				X^{k\ell}Z^{km}X^{\ell'}Z^{m'}=\om^{k\ell' m-k\ell m'}X^{\ell'}Z^{m'}X^{k\ell}Z^{km}
			\end{equation*}
			with $\om^{k\ell' m-k\ell m'}\neq 1$. This, together with Lemma \ref{lem:orthogonal}, implies
			\begin{equation*}
				\tr[X^{k\ell}Z^{km}\ketbra{e^{\ell',m'}_{z}}{e^{\ell',m'}_{z}}]
				=\langle e^{\ell',m'}_{z},X^{k\ell}Z^{km}e^{\ell',m'}_{z}\rangle=0, \qquad z\in \Om_\go.
			\end{equation*}
			
			All combined, for all $1\le k\le \go-1,(\ell,m)\in \Sigma_\go$ and $1\le j\le n$ we get
			\begin{align*}
				\tr[X^{k\ell} Z^{k m } \rho_j(\vec{\om})]
				&=\frac{1}{\go+1}\sum_{(\ell',m')\in \Sigma_\go}\left\langle e^{\ell',m'}_{\om^{\ell',m'}_j}, X^{k\ell}Z^{ km}e^{\ell',m'}_{\om^{\ell',m'}_j}\right\rangle\\
				&=\frac{1}{\go+1}\left\langle e^{\ell,m}_{\om^{\ell,m}_j}, X^{k\ell}Z^{ km}e^{\ell,m}_{\om^{\ell,m}_j}\right\rangle\\
				&=\frac{1}{\go+1}\om^{-\frac{1}{2}k(k-1)\ell m}(\om^{\ell,m}_j)^k.
			\end{align*}

			Note that by Lemma \ref{lem:HW basis prime} any polynomial in $M_\go(\C)^{\otimes n}$ of degree at most $d$ is a linear combination of monomials
			\[
			A(\vec{k},\vec{\ell},\vec{m};\vec{i}):=
			\cdots \otimes X^{k_1\ell_1}Z^{k_1 m_1}\otimes \cdots \otimes X^{k_\kappa\ell_\kappa}Z^{k_\kappa m_\kappa}\otimes\cdots
			\]
			where
			\begin{itemize}
				\item $\vec{k}=(k_1,\dots,k_\kappa)\in \{1,\dots, \go-1\}^\kappa$ with $0\le \sum_{j=1}^{\kappa}k_j\le d$;
				\item $\vec{\ell}=(\ell_1,\dots, \ell_\kappa),\vec{m}=(m_1,\dots, m_\kappa)$ with each $(\ell_j,m_j)\in \Sigma_\go$;
				\item $\vec{i}=(i_1,\dots, i_\kappa)$ with $1\le i_1 <\cdots<i_\kappa\le n$;
				\item $X^{k_j\ell_j}Z^{k_j m_j}$ appears in the $i_j$-th place, $1\le j\le \kappa$, and all the other $n-\kappa$ elements in the tensor product are the ($\go\times \go$) identity matrices $\un$.
			\end{itemize}
			So for any $\vec{\om}\in \Om_\go^{(\go+1)n}$ of the form \eqref{eq:defn of vec omega} we have from the above discussion that
			\begin{align*}
				\tr[A(\vec{k},\vec{\ell},\vec{m};\vec{i})\rho(\vec{\om})]
				&=\prod_{j=1}^{\kappa}\tr[X^{k_j\ell_j}Z^{k_j m_j}\rho_{i_j}(\vec{\om})]\\
				&=\frac{\om^{-\frac{1}{2}\sum_{j=1}^{\kappa}k_j(k_j-1)\ell_j m_j}}{(\go+1)^{\kappa}}(\om^{\ell_1,m_1}_{i_1})^{k_1}\cdots (\om^{\ell_\kappa,m_\kappa}_{i_\kappa})^{k_\kappa}.
			\end{align*}
			Thus $\vec{\om}\mapsto 	\tr[A(\vec{k},\vec{\ell},\vec{m};\vec{i})\rho(\vec{\om})]$ is a monomial on $\Om_\go^{(\go+1)n}$ of degree at most $\sum_{j=1}^{\kappa}k_j\le d$.
			
			Now for general $A\in M_\go(\C)^{\otimes n}$ of degree at most $d$:
			\begin{equation*}
				A=\sum_{\vec{k},\vec{\ell},\vec{m},\vec{i}} c(\vec{k},\vec{\ell},\vec{m};\vec{i})A(\vec{k},\vec{\ell},\vec{m};\vec{i})
			\end{equation*}
			where the sum runs over the above $(\vec{k},\vec{\ell},\vec{m};\vec{i})$. This is the Fourier expansion of $A$ and $\{c(\vec{k},\vec{\ell},\vec{m};\vec{i})\}$ are the Fourier coefficients. So 
			\begin{equation*}
				\|\widehat{A}\|_p=\left(\sum_{\vec{k},\vec{\ell},\vec{m},\vec{i}}|c(\vec{k},\vec{\ell},\vec{m};\vec{i})|^p\right)^{1/p},\qquad p>0.
			\end{equation*}
			To each $A$ we assign the function $f_A$ on $\Om_\go^{(\go+1)n}$ given by
			\begin{align*}
				f_A(\vec{\om})&=\tr[A\rho(\vec{\om})]\\
				&=\sum_{\vec{k},\vec{\ell},\vec{m},\vec{i}} \frac{\om^{-\frac{1}{2}\sum_{j=1}^{\kappa}k_j(k_j-1)\ell_j m_j}c(\vec{k},\vec{\ell},\vec{m};\vec{i})}{(\go+1)^{\kappa}}(\om^{\ell_1,m_1}_{i_1})^{k_1}\cdots (\om^{\ell_\kappa,m_\kappa}_{i_\kappa})^{k_\kappa}.
			\end{align*}
			Note that this is the Fourier expansion of $f_A$ since the monomials $(\om^{\ell_1,m_1}_{i_1})^{k_1}\cdots (\om^{\ell_\kappa,m_\kappa}_{i_\kappa})^{k_\kappa}$ differ for distinct $(\vec{k},\vec{\ell},\vec{m};\vec{i})$'s. Therefore, for $p>0$
			\begin{align*}
				\|\widehat{f_A}\|_p
				=\left(\sum_{\vec{k},\vec{\ell},\vec{m},\vec{i}}\left|\frac{c(\vec{k},\vec{\ell},\vec{m};\vec{i})}{(\go+1)^{\kappa}}\right|^p\right)^{1/p}
				\ge \frac{1}{(\go+1)^{d}}\left(\sum_{\vec{k},\vec{\ell},\vec{m},\vec{i}}|c(\vec{k},\vec{\ell},\vec{m};\vec{i})|^p\right)^{1/p}
				=\frac{1}{(\go+1)^{d}}	\|\widehat{A}\|_p.
			\end{align*}
			According to Theorem \ref{thm:bh cyclic qudit}, one has
			\begin{equation*}
				\|\widehat{f_A}\|_{\frac{2d}{d+1}}\le \textnormal{BH}^{\le d}_{\Omega_\go}\|f_A\|_{\Omega_\go^{(\go+1)n}}
			\end{equation*}
			for some $\textnormal{BH}^{\le d}_{\Omega_\go}<\infty$. Since each $\rho(\vec{\omega})$ is a density matrix, we have by duality that 
			\begin{equation*}
				\|f_A\|_{\Omega_\go^{(\go+1)n}}=\sup_{\vec{\omega}\in \Omega_{\go}^{(\go+1)n}}|\tr[A\rho(\vec{\omega})]|\le \|A\|.
			\end{equation*}
			All combined, we obtain
			\begin{equation*}
				\|\widehat{A}\|_{\frac{2d}{d+1}}
				\le (\go+1)^d	\|\widehat{f_A}\|_{\frac{2d}{d+1}}
				\le (\go+1)^d\textnormal{BH}^{\le d}_{\Omega_\go}\|f_A\|_{\Omega_\go^{(\go+1)n}}
				\le (\go+1)^d\textnormal{BH}^{\le d}_{\Omega_\go}\|A\|\,. \qedhere
			\end{equation*}
		\end{proof}
		

		
		
		So the reduction method still works for qudit systems (for prime $\go$), but we still need to prove the classical Bohnenblust--Hille inequality for cyclic groups, otherwise the above reduction does not give anything meaningful. This question will be addressed in the next lecture.

		\chapter{Bohnenblust--Hille inequality on general cyclic groups and more dimension-free phenomena} 
		
		\section{An obstacle: a generalized maximum modulus principle}
		
		For $\go\ge 3$ and $n\ge 1$, any function $f:\Omega_\go^n\to \C$ has the unique Fourier expansion 
		\begin{equation}
			f(z)=\sum_{\alpha\in \{0,1,\dots, \go-1\}^n}\widehat{f}(\alpha)z^{\alpha},\qquad z\in \Omega_\go^n.
		\end{equation}
		Here, we fix each $\alpha_j\in \{0,1,\dots, \go-1\}$ so that we can extend $f$ to an analytic function on $\C^n$ (which we still denote by $f$ in the following) with the same expression. We say that $f$ is degree at most $d$ if $\widehat{f}(\alpha)=0$ whenever $|\alpha|=\sum_{j}\alpha_j>d.$
		
		In this lecture, we aim to prove 
		
		\begin{theorem}\label{thm: bh cyclic}
			Fix $d\ge 1$ and $\go\ge 3$. There exists $C(d,\go)>0$ such that for all $n\ge 1$ and for all $f:\Omega_\go^n\to \C$ of degree at most $d$ we have 
			\begin{equation}\label{ineq: bh cyclic}
				\|\widehat{f}\|_{\frac{2d}{d+1}}\le C(d,\go)\|f\|_{\Omega_\go^n}.
			\end{equation}
			Moreover, for the best constant $\bh^{\le d}_{\Omega_\go}$, we have $\bh^{\le d}_{\Omega_\go}\le C_\go^{d^2}$ for prime $\go$.
		\end{theorem}

		
		The proofs of Bohnenblust--Hille inequalities for circle groups \eqref{ineq:bh for torus}  and Boolean cubes \eqref{ineq:boolean bh} (can be understood as cases $\go=\infty$ and $\go=2$ respectively) are similar. If one follows the same four-step argument presented in the first lecture, then one may find an obstacle in the \textbf{step 4} using polarization where a convex structure is needed. After a careful adaption of the first three steps, one may end up with ($\conv\Omega_\go$ denoting the convex hull of $\Omega_\go$)
		\begin{equation}\label{ineq:bh convex hull}
			\|\widehat{f}\|_{\frac{2d}{d+1}}\le C(d,\go)\|f\|_{(\conv\Omega_\go)^{n}}
		\end{equation}
		for all $f:\Omega_\go^{n}\to \C$ of degree at most $d$, which is not exactly what we need. A quicker way to achieve the same inequality \eqref{ineq:bh convex hull}, regardless of the constant, is to embed the unit circle/disc into the $\conv\Omega_\go$ after scaling. In fact, let $r=r_\go=\cos^{-1}(\pi/\go)\in (1,2]$ which is the smallest $r>0$ such that $r\T\subset \conv(\Om_\go)$. Then $g(\cdot)=f(\cdot/r_\go)$ is also of degree at most $d$, so
		\begin{align}
			\|\widehat{f}\|_{\frac{2d}{d+1}}
			\le r_\go^d \left(\sum_{|\alpha|\le d}\left|\widehat{f}(\alpha)/r_\go^{|\alpha|}\right|^{\frac{2d}{d+1}}\right)^{\frac{d+1}{2d}}
			= r_\go^d\|\widehat{g}\|_{\frac{2d}{d+1}}
			&\le r_\go^d\bh^{\le d}_{\T}\|g\|_{\T^n}\\
			&=r_\go^d\bh^{\le d}_{\T}\|f\|_{(r_\go\T)^{n}}
			\le  r_\go^d\bh^{\le d}_{\T}\|f\|_{(\conv\Omega_\go)^{n}}
		\end{align}
		where we used the Bohnenblust--Hille inequality for $\T^n$ \eqref{ineq:bh for torus}. 
		
		So it seems that we need to prove the following dimension-free estimate
		\begin{equation}\label{ineq:obstacle of convex hull}
			\|f\|_{(\conv\Omega_\go)^{n}}
			\le C(d,\go)\|f\|_{\Omega_\go^{n}}
		\end{equation}
		for all $f:\Omega_\go^{n}\to \C$ of degree at most $d$. When $\go=2$, this becomes an equality with $C(d,\go)=1$ by convexity since all $f$ in consideration are multi-affine. When $\go=\infty$ meaning that $\Om_\infty=\T$, this is also an equality with $C(d,\go)=1$ by the maximum modulus principle. 
		

		However, for $3\le \go<\infty$, this \emph{generalized maximum modulus principle} looks highly non-trivial. Already in the one-dimensional case, the constant cannot be 1, as we shall explain now for $\go=3$. Let $\omega:=e^{\frac{2\pi i}{3}}$. Consider the polynomial
		\begin{equation*}
			p(z):=p(1)\frac{(z-\omega)(z-\omega^2)}{(1-\omega)(1-\omega^2)}
			+p(\omega)\frac{(z-1)(z-\omega^2)}{(\omega-1)(\omega-\omega^2)}
			+p(\omega^2)\frac{(z-1)(z-\omega)}{(\omega^2-1)(\omega^2-\omega)},\qquad z\in \Omega_3
		\end{equation*}
		with $p(1), p(\omega), p(\omega^2)$ to be chosen later. Put $z_0:=\frac{1+\omega}{2}\in \conv(\Omega_3)$. Then 
		\begin{equation*}
			|z_0-1|=|z_0-\omega|=\frac{\sqrt{3}}{2},\qquad |z_0-\omega^2|=\frac{3}{2}.
		\end{equation*}
		Now we choose $p(1), p(\omega), p(\omega^2)$ to be complex numbers of modulus $1$ such that 
		\begin{equation*}
			p(1)\frac{(z_0-\omega)(z_0-\omega^2)}{(1-\omega)(1-\omega^2)}=\left|\frac{(z_0-\omega)(z_0-\omega^2)}{(1-\omega)(1-\omega^2)}\right|
			=\frac{\frac{3\sqrt{3}}{4}}{3}
			=\frac{\sqrt{3}}{4},
		\end{equation*}
		\begin{equation*}
			p(\omega)\frac{(z_0-1)(z_0-\omega^2)}{(\omega-1)(\omega-\omega^2)}
			=\left|\frac{(z_0-1)(z_0-\omega^2)}{(\omega-1)(\omega-\omega^2)}\right|
			=\frac{\frac{3\sqrt{3}}{4}}{3}
			=\frac{\sqrt{3}}{4},
		\end{equation*}
		\begin{equation*}
			p(\omega^2)\frac{(z_0-1)(z_0-\omega)}{(\omega^2-1)(\omega^2-\omega)}
			=\left|\frac{(z_0-1)(z_0-\omega)}{(\omega^2-1)(\omega^2-\omega)}\right|
			=\frac{\frac{3}{4}}{3}
			=\frac{1}{4}.
		\end{equation*}
		Therefore, this choice of $p$ satisfies 
		\begin{equation*}
			\|p\|_{\conv(\Omega_3)}\ge |p(z_0)|
			=\frac{\sqrt{3}}{4}+\frac{\sqrt{3}}{4}+\frac{1}{4}
			=\frac{1+2\sqrt{3}}{4}>1=\|p\|_{\Omega_3}.
		\end{equation*}
		Therefore, a simple tensorization argument will not give \eqref{ineq:obstacle of convex hull}, since the constant will blow up as $n\to\infty$. So we do need to make use of the low-degree assumption. 
		

		Now we are faced with two approaches of proving Theorem \ref{thm: bh cyclic}: (1) overcome the obstacle \eqref{ineq:obstacle of convex hull}; or (2) find a way to bypass \eqref{ineq:obstacle of convex hull} to prove cyclic Bohnenblust--Hille inequality \eqref{ineq: bh cyclic} directly.  We will start with the approach (2) but its solution also led us to overcome the obstacle \eqref{ineq:obstacle of convex hull} as well. 
		
		\section{ Bohnenblust--Hille inequality on general cyclic groups: proof sketch}
		
		In this section we sketch a proof of Theorem \ref{thm: bh cyclic} without knowing \eqref{ineq:obstacle of convex hull}. We refer to \cite{svz24bh} for details. To avoid some technical issues and to be consistent with the reduction method presented in the last lecture, we consider prime $\go \ge 3$ only.
		Our idea is to reduce the Bohnenblust--Hille inequality \eqref{ineq: bh cyclic} over $\Om_\go^n,\go\ge 3$ to that over $\Om_2=\{-1,1\}^n$ \eqref{ineq:boolean bh} which may sound unreasonable. To grasp the idea, we start with a simple observation 
		\begin{equation}
			\|f\|_{\Om_\go^n}\ge \|f\|_{\{a,b\}^n}
		\end{equation}
		for any $a\neq b\in \Om_\go$. One may think of $\{a,b\}$ as an affine transformation of $\{-1,1\}$ via
		\begin{equation}
			\left\{\frac{a+b}{2}+\frac{a-b}{2}x,x=\pm1 \right\}=\{a,b\}.
		\end{equation}
		and then we can apply  known dimension-free estimates over $\{-1,1\}^n.$ 
		
		So for any $A\subset [n]$ with $|A|=m$, and for each $\alpha$ with $\supp(\alpha)=A$, we have for $z\in \{a,b\}^n$:
		\begin{align*}
			z^{\alpha}=\prod_{j:\alpha_j\neq 0}z_j^{\alpha_j}
			=\prod_{j:\alpha_j\neq 0}\left(\frac{a^{\alpha_j}+b^{\alpha_j}}{2}+\frac{a^{\alpha_j}-b^{\alpha_j}}{2}x_j\right)
			=&\prod_{j:\alpha_j\neq 0}\left(\frac{a^{\alpha_j}-b^{\alpha_j}}{2}\right)\cdot x^A+\cdots
		\end{align*}
		where $x_A$ is of degree $|A|=m$ while $\cdots$ is of degree $<m$. For convenience, we put 
		\begin{equation}
			\tau_\alpha^{(a,b)}=\prod_{j:\alpha_j\neq 0}\left(a^{\alpha_j}-b^{\alpha_j}\right).
		\end{equation}		
		So when restricted to $\{a,b\}^n$, the polynomial $f(z)=\sum_{|\supp(\alpha)|\le \ell}a_{\alpha}z^\alpha$ becomes 
		\begin{equation}\label{eq: z to x}
			f(z)
			=\sum_{m\le \ell}\frac{1}{2^m}\sum_{|A|=m}\left(\sum_{ \supp(\alpha)=A} a_{\alpha}\tau_\alpha^{(a,b)}\right)
			x_A+\cdots,\qquad x\in \Omega_2^n.
		\end{equation}
		Again, for each $m\le \ell$, $\cdots$ is some polynomial of degree $<m$. To summarize, one gets a low-degree function \eqref{eq: z to x} on $\Omega_2^n$ after restricting to (product of) two-point subsets. In general, the ... part looks very complicated but its highest degree level has a good form and carries some information of the largest support level of the original function $f$. We shall have a surgery starting from here, before which we need the following lemma that can be found in \cite{dmp19}.
		
		%

		
		\begin{proposition}\label{prop:riesz proj}
			Fix $1\le d\le n$. For any  $g:\Omega_2^n\to \C$ of degree at most $d$, denote by $g_m$ its $m$-homogeneous part, $0\le m\le d$. Then 
			\begin{equation}
				\|g_m\|_{\Omega_2^n}\le (1+\sqrt{2})^d\|g\|_{\Omega_2^n},\qquad 0\le m\le d.
			\end{equation}		 
		\end{proposition}

		\begin{definition}[Support-homogeneous polynomials]
			A polynomial $f:\Om_K^n\to\C$ is $\ell$-\emph{support-homogeneous} if it can be written as
			\[f(z) = \sum_{\alpha:|\supp(\alpha)|=\ell}a_\alpha z^\alpha\,.\]
		\end{definition}
		
		We will employ a certain operator that removes monomials whose support sizes are not maximal and alters the coefficients of the remaining terms (those of maximal support size) in a controlled way.
		
		\begin{definition}[Maximal support pseudo-projection]
			Fix $\xi\in \Om_K\backslash\{1\}$.
			For any multi-index $\alpha\in\{0,1,\ldots, K-1\}^n$ we define the factor
			\begin{equation}
				\tau_\alpha^{(\xi)} =\prod_{j: \alpha_j\neq 0}(1-\xi^{\alpha_j})\,.
				\label{eq:tau}
			\end{equation}
			For any polynomial with maximal support size $\ell\ge 0$,
			$$f(z)=\sum_{|\supp(\alpha)|\le \ell}a_\alpha z^\alpha,$$
			on $\Omega_K^n$, we define $\mathfrak{D}_{\xi} f:\Omega_K^n\times\Omega_2^n\to \C$ via
			\[\mathfrak{D}_{\xi}f(z,x) =\sum_{|\supp(\alpha)|= \ell}\tau_\alpha^{(\xi)}\,a_\alpha z^\alpha x^{\supp(\alpha)}\,.\]
		\end{definition}
		
		Denote by $\vec{1}=(1,\dots, 1)$ the vector in $\Omega_2^n$ that has all entries $1$.  Note that $\mathfrak{D}_{\xi}f(\cdot,\vec 1)$ is the $\ell$-support-homogeneous part of $f$, except where the coefficients $a_\alpha$ have picked up the factor $\tau_\alpha^{(\xi)}$.
		And we note the relationships among the $\tau_\alpha^{(\xi)}$'s can be quite intricate; while in general they are different for distinct $\alpha$'s, this is not always true:
		consider the case of $K=3$ and the two monomials
		\begin{equation}
			\label{ex:same-taus}
			z^{\beta}:=z_1^2 z_2z_3z_4z_5z_6z_7z_8,\quad z^{\beta'}:=z_1^2 z_2^2 z_3^2 z_4^2 z_5^2 z_6^2 z_7^2z_8\,.
		\end{equation}
		Then
		\[
		\tau_{\beta}^{(\omega)}\;=\;(1-\omega)^7(1-\omega^2)\;=\;(1-\omega)(1-\omega^2)^7\;=\;\tau^{(\omega)}_{\beta'},
		\]
		which follows from the identity $(1-\omega)^6=(1-\omega^2)^6$ for $\omega=e^{\frac{2\pi i}{3}}$.
		
		\medskip
		
		The somewhat technical definition of $\mathfrak{D}_{\xi}$ is motivated by the proof of its key property, namely that it is bounded from $L^\infty(\Omega_K^n)$ to $L^\infty({\Om_K^n\times\Om_2^n})$ with a dimension-free constant when restricted to low-degree (actually low support size) polynomials. 
		
		\begin{proposition}[Dimension-free boundedness of $\mathfrak{D}_{\xi}$]\label{prop:bdd of d}
			
			Let $f:\Omega_K^n\to\C$ be a polynomial of degree at most $d$ and $\ell$ be the maximum support size of monomials in $f$.
			Then for all $\xi \in \Omega_K\backslash\{1\}$,
			\begin{equation}
				\label{ineq:main ingredient}
				\|\mathfrak{D}_{\xi}f\|_{\Om_K^n\times\Om_2^n}\le (2+2\sqrt{2})^{\ell} \|f\|_{\Omega_K^n}.
			\end{equation}
		\end{proposition}
		\begin{proof}
			Consider the operator $G_{\xi}$:
			\begin{equation*}
				G_{\xi}(f)(x)
				=f\left(\frac{1+\xi}{2}+\frac{1-\xi}{2}x_1,\dots, \frac{1+\xi}{2}+\frac{1-\xi}{2}x_n\right),\qquad x\in\Omega_2^n
			\end{equation*}
			that maps any function $f:\{1,\xi\}^n\subset \Om_K^n\to\C$ to a function $G_{\xi}(f):\Omega_2^n\to\C$.
			Then by definition
			\begin{equation}\label{ineq:G compare K prime}
				\|f\|_{\Om_K^n}\ge 	\|f\|_{\{1,\xi\}^n}=\|G_{\xi}(f)\|_{\Omega_2^n}. 
			\end{equation}
			Fix $m\le \ell$. For any $\alpha$ we denote 
			\begin{equation*}
				m_k(\alpha):=|\{j:\alpha_j=k\}|,\qquad 0\le k\le K-1.
			\end{equation*}
			Then for $\alpha$ with $|\supp(\alpha)|=m$, we have
			\begin{equation}
				m_1(\alpha)+\cdots+m_{K-1}(\alpha)=|\supp(\alpha)|=m.
			\end{equation}
			For $z\in \{1,\xi\}^n$ with $z_j=\frac{1+\xi}{2}+\frac{1-\xi}{2}x_j,$ $x_j=\pm 1$, note that 
			\begin{equation*}
				z_j^{\alpha_j}=\left(\frac{1+\xi}{2}+\frac{1-\xi}{2}x_j \right)^{\alpha_j}=\frac{1+\xi^{\alpha_j}}{2}+\frac{1-\xi^{\alpha_j}}{2}x_j\,.
			\end{equation*}
			So for any $A\subset [n]$ with $|A|=m$, and for each $\alpha$ with $\supp(\alpha)=A$, we have for $z\in \{1,\xi\}^n$:
			\begin{align*}
				z^{\alpha}
				=\prod_{j:\alpha_j\neq 0}z_j^{\alpha_j}
				=\prod_{j:\alpha_j\neq 0}\left(\frac{1+\xi^{\alpha_j}}{2}+\frac{1-\xi^{\alpha_j}}{2}x_j\right)
				=\prod_{j:\alpha_j\neq 0}\left(\frac{1-\xi^{\alpha_j}}{2}\right)\cdot x_A+\cdots
				=2^{-m}\tau_\alpha^{(\xi)} x_A+\cdots
			\end{align*}
			where $x_A$ is of degree $|A|=m$ while $\cdots$ is of degree $<m$. Then for $f(z)=\sum_{|\supp(\alpha)|\le \ell}a_{\alpha}z^\alpha$ 
			\begin{align*}
				G_{\xi}(f)(x)
				=\sum_{m\le \ell}\frac{1}{2^m}\sum_{|A|=m}\left(\sum_{ \supp(\alpha)=A} a_{\alpha}\tau_\alpha^{(\xi)}\right)
				x_A+\cdots,\qquad x\in \Omega_2^n.
			\end{align*}
			Again, for each $m\le \ell$, $\cdots$ is some polynomial of degree $<m$. So $G_{\xi}(f)$ is of degree $\le \ell$ and the $\ell$-homogeneous part is nothing but 
			\begin{equation*}
				\frac{1}{2^\ell}\sum_{|A|=\ell}\left(\sum_{ \supp(\alpha)=A} \tau_\alpha^{(\xi)}a_{\alpha}\right)
				x_A.
			\end{equation*}
			Consider the projection operator $Q$ that maps any polynomial on $\Omega_2^n$ onto its highest level homogeneous part; \emph{i.e.}, for any polynomial $g:\Omega_2^n\to\C$ with $\deg(g)=m$ we denote $Q(g)$ its $m$-homogeneous part.
			Then we just showed that 
			\begin{align*}
				Q(G_{\xi}(f))(x)	
				=\frac{1}{2^\ell}\sum_{|A|=\ell}\left(\sum_{ \supp(\alpha)=A} \tau_\alpha^{(\xi)}a_{\alpha}\right)
				x_A.
			\end{align*}
			According to Proposition \ref{prop:riesz proj} and \eqref{ineq:G compare K prime}, we have
			\begin{equation*}
				\|Q(G_{\xi}(f))\|_{\Omega_2^n}\le (1+\sqrt{2})^{\ell}\|G_{\xi}(f)\|_{\Omega_2^n}
				\le (1+\sqrt{2})^{\ell} \|f\|_{\Omega_K^n}
			\end{equation*}
			and thus 
			\begin{align*}
				\left\| \sum_{|A|=\ell}\left(\sum_{ \supp(\alpha)=A} \tau_\alpha^{(\xi)} a_{\alpha}\right)
				x_A\right\|_{\Omega_2^n}
				\le(2+2\sqrt{2})^{\ell}\|f\|_{\Omega_K^n}.
			\end{align*}
			
			The left-hand side is almost $\mathfrak{D}_{\xi}f$.
			Recall that $\Omega_K^n$ is a group, so we have 
			\begin{equation*}
				\sup_{z,\zeta\in\Om_K^n}\left|\sum_{\alpha}a_{\alpha}z^{\alpha}\zeta^{\alpha}\right|
				=\sup_{z\in\Om_K^n}\left|\sum_{\alpha}a_{\alpha}z^{\alpha}\right|.
			\end{equation*}
			Thus actually we have shown
			\begin{equation}
				\sup_{z\in\Om_K^n,x\in \Omega_2^n}\left| \sum_{|A|=\ell}\left(\sum_{ \supp(\alpha)=A} \tau^{(\xi)}_\alpha a_{\alpha}z^{\alpha}\right)
				x_A\right|	\le (2+2\sqrt{2})^{\ell} \|f\|_{\Omega_K^n},
			\end{equation}
			which is exactly \eqref{ineq:main ingredient}.
		\end{proof}

		The second set of variables in $\mathfrak{D}_{\xi}f$ are key to proving the support-homogeneous case of the cyclic Bohnenblust--Hille inequalities.
		
		\begin{lemma}[Support-homogeneous cyclic Bohnenblust--Hille inequalities]
			\label{lem:support-homo-BH}
			Let $f_\ell:\Om_K^n\to\C$ be an $\ell$-support-homogeneous polynomial of degree at most $d$.
			Then
			\[\|\widehat{f_\ell}\|_{\frac{2d}{d+1}}\leq C_{K}^d\|f_\ell\|_{\Om_K^n}.\]
		\end{lemma}
		\begin{proof}
			Let $f_\ell(z) = \sum_{|\supp(\alpha)|=\ell}a_\alpha z^\alpha$. For any $z\in \Omega_K^n$,
			apply the Boolean Bohnenblust--Hille inequality \eqref{ineq:boolean bh} for $\Omega_2^n$ to $x\mapsto \mathfrak{D}_{\omega}f_\ell(z,x)$ and  \eqref{ineq:main ingredient} for $(f;\xi)=(f_\ell;\omega)$ to obtain
			\begin{equation*}
				\sum_{|A|=\ell}\left|\sum_{ \supp(\alpha)=A} \tau^{(\omega)}_\alpha a_{\alpha}z^{\alpha}\right|^{\frac{2d}{d+1}}
				\le\left[ \BH^{\le d}_{\Om_2}(2+2\sqrt{2})^{\ell} \|f_\ell\|_{\Omega_K^n}\right]^{\frac{2d}{d+1}}.
			\end{equation*}
			Taking the expectation with respect to the Haar measure over $z\sim \Omega_K^n$, we get
			\begin{equation}\label{ineq:expectation}
				\sum_{|A|=\ell}\E_{z\sim  \Omega_K^n}\left|\sum_{ \supp(\alpha)=A} \tau^{(\omega)}_\alpha a_{\alpha}z^{\alpha}\right|^{\frac{2d}{d+1}}
				\le\left[ \BH^{\le d}_{\Om_2}(2+2\sqrt{2})^{\ell} \|f_\ell\|_{\Omega_K^n}\right]^{\frac{2d}{d+1}}.
			\end{equation}
			Note that for fixed $A\subset[n]$ with $|A|=\ell$ and for each $\alpha$ with $\supp(\alpha)=A$
			\begin{align*}
				\big|\tau^{(\omega)}_\alpha a_{\alpha}\big|^{\frac{2d}{d+1}}
				\le  \left(\E_{z\sim  \Omega_K^n}\left|\sum_{ \supp(\alpha)=A} \tau^{(\omega)}_\alpha a_{\alpha}z^{\alpha}\right|\right)^{\frac{2d}{d+1}}
				\le \E_{z\sim  \Omega_K^n}\left|\sum_{ \supp(\alpha)=A} \tau^{(\omega)}_\alpha a_{\alpha}z^{\alpha}\right|^{\frac{2d}{d+1}},
			\end{align*}
			where the first inequality uses Hausdorff--Young inequality and the second inequality follows from H\"older's inequality. The number of all such $\alpha$'s is bounded by (recalling $A$ is fixed)
			\begin{equation*}
				|\{\al: \supp(\al) =A\}| \le (K-1)^\ell\le (K-1)^d,
			\end{equation*}
			thus 
			\begin{align*}
				\sum_{ \supp(\alpha)=A}|\tau^{(\omega)}_\alpha|^{\frac{2d}{d+1}}\cdot |a_{\alpha}|^{\frac{2d}{d+1}}
				\;\le\; (K-1)^d\E_{z\sim  \Omega_K^n}\left|\sum_{ \supp(\alpha)=A} \tau^{(\omega)}_\alpha  a_{\alpha}z^{\alpha}\right|^{\frac{2d}{d+1}}\,.
			\end{align*}
			This, together with  \eqref{ineq:expectation} and the fact that $\BH^{\le d}_{\Omega_2}\le C^d$, yields
			\begin{align*}
				\left(\sum_{|A|=\ell}\sum_{ \supp(\alpha)=A}	|\tau^{(\omega)}_\alpha|^{\frac{2d}{d+1}}\cdot |a_{\alpha}|^{\frac{2d}{d+1}}\right)^{\frac{d+1}{2d}}
				\le C_K^d \|f_\ell\|_{\Omega_K^n},
			\end{align*}
			for some $C_K>0$. Now we bound $|\tau^{(\omega)}_\alpha|$ from below with
			\begin{equation*}
				\big|\tau^{(\omega)}_\alpha\big|=\prod_{j:\alpha_j\neq 0}|1-\omega^{\alpha_j}|\ge |1-\omega|^{\ell}=(2\sin(\pi/K))^\ell\ge (\sin(\pi/K))^d>0\,,
			\end{equation*}
			and therefore, 
			\begin{align*}
				\left(\sum_{|A|=\ell}\sum_{ \supp(\alpha)=A}	|a_{\alpha}|^{\frac{2d}{d+1}}\right)^{\frac{d+1}{2d}}
				\le\left( \frac{C_K}{\sin(\pi/K)}\right)^d \|f_\ell\|_{\Omega_K^n}
				=C_K^d\|f_\ell\|_{\Omega_K^n}. \tag*{\qedhere}
			\end{align*}
		\end{proof}
		
		With the support-homogeneous Cyclic Bohnenblust--Hille inequality proved, it remains to reduce the proof of Bohnenblust--Hille inequalities \eqref{ineq: bh cyclic} from the general case to the support-homogeneous case.
		This is done by proving the following lemma as an analog of Proposition \ref{prop:riesz proj} for general cyclic groups $\Omega_K^n$.
		
		\begin{lemma}[Splitting Lemma]
			\label{lem:splitting-lemma}
			Let $f:\Om_K^n\to \C$ be a polynomial of degree at most $d$ and for $0\le j\le d$ let $f_j$ be its $j$-support-homogeneous part.
			Then for all $0\le j\le d$, 
			\begin{equation}\label{ineq:splitting lemma}
				\|f_j\|_{\Om_K^n}\leq C(d,\go)\|f\|_{\Om_K^n}
			\end{equation}
			where $C(d,\go)$ is a constant independent of $n$.
			Moreover, for $K$ prime we may take $C(d,\go)$ to be $C_K^{d^2}$.
		\end{lemma}
		
		We pause to note Lemmas \ref{lem:support-homo-BH} and \ref{lem:splitting-lemma} together immediately give the full Cyclic Bohnenblust--Hille inequality:
		
		
		\begin{proof}[Proof of Theorem \ref{thm: bh cyclic}]
			For $0\le j\le d$ let $f_j$ be the $j$-support-homogeneous part of $f$. The triangle inequality and our lemmas give
			\begin{align*}
				\|\widehat{f}\|_{\frac{2d}{d+1}}\;\leq\; \sum_{0\le j\le d}\|\widehat{f_j}\|_{\frac{2d}{d+1}}\overset{\text{Lemma \ref{lem:support-homo-BH}}}{\Lesssim{K,d}}\sum_{0\le j\le d}\|f_j\|_{\Omega_K^n}\overset{\text{Lemma \ref{lem:splitting-lemma}}}{\Lesssim{K,d}}\|f\|_{\Omega_K^n}\,.
			\end{align*}
			Estimates of $C(d,\go)$ easily follow from the lemma statements.
		\end{proof}
		

		For simplicity of notation we shall denote $\mathfrak{D}_{\xi}f(\cdot,\vec1)$ by $\mathfrak{D}_{\xi}f(\cdot):\Om_K^n\to\C$ unless otherwise stated.
		Concretely, in what follows we have
		\begin{equation*}
			\mathfrak{D}_{\xi}f(z):=\mathfrak{D}_{\xi}f(z,\vec1)= \sum_{|\supp(\alpha)|=\ell}\tau_\alpha^{(\xi)}\,a_\alpha z^\alpha \,
		\end{equation*}
		for polynomials with largest support size $\ell$: 
		$$
		f(z)=\sum_{|\supp(\alpha)|\le \ell}a_\alpha z^\alpha, \qquad a_{\alpha}\neq 0 \textnormal{ for some } |\supp(\alpha)|= \ell.
		$$
		
		\begin{proof}[Proof of Lemma \ref{lem:splitting-lemma} for prime $K$]
			Let $\ell\le d$ be the largest support size of monomials in $f$. For any $m\le \ell$, we denote by $f_m$ the $m$-support homogeneous part of $f$. Then we shall prove the lemma for $f_\ell$ first.
			
			To show \eqref{ineq:splitting lemma} for $f_\ell$, recall that \eqref{ineq:main ingredient} gives
			\begin{align*}
				\|\mathfrak{D}_{\xi}f\|_{\Om_K^n}=\left\| \sum_{ |\supp(\alpha)|=\ell} \tau_\alpha^{(\xi)}a_{\alpha}z^{\alpha}\right\|_{\Omega_K^n}
				\le (2+2\sqrt{2})^{\ell}\|f\|_{\Omega_K^n}\,.
			\end{align*}
			We would like to replace the polynomial on the left-hand side with $f_\ell$, and the main issue is that the factors
			\begin{equation*}
				\tau_\alpha^{(\xi)}=\prod_{j:\alpha_j\neq 0}\left(1-\xi^{\alpha_j}\right),
			\end{equation*}
			may differ for different $\alpha$'s under consideration as we discussed before. To overcome this difficulty, we rotate $\xi$ over repeated applications of $\mathfrak{D}_{\xi}$.
			This will lead to an accumulated factor that is constant across all monomials in $f_\ell$.
			Begin by considering $\mathfrak{D}_{\omega}f$ to obtain 
			\begin{align*}
				\left\| \sum_{ |\supp(\alpha)|=\ell} \tau_\alpha^{(\omega)}a_{\alpha} z^{\alpha}\right\|_{\Omega_K^n}
				\le (2+2\sqrt{2})^{\ell}\|f\|_{\Omega_K^n}.
			\end{align*}
			Then we apply $\mathfrak{D}_{\omega^2}$ to $\mathfrak{D}_{\omega}f$ to obtain 
			\begin{align*}
				\left\| \sum_{ |\supp(\alpha)|=\ell} \tau_\alpha^{(\omega)}\tau_\alpha^{(\omega^2)}a_{\alpha} z^{\alpha}\right\|_{\Omega_K^n}
				\le (2+2\sqrt{2})^{\ell}\|\mathfrak{D}_{\omega}f\|_{\Omega_K^n}
				\le (2+2\sqrt{2})^{2\ell}\|f\|_{\Omega_K^n}.
			\end{align*}
			We continue iteratively applying $\mathfrak{D}_{\omega^k}$ to $\mathfrak{D}_{\omega^{k-1}}\cdots\mathfrak{D}_{\omega}f$ and finally arrive at
			\begin{align*}
				\left\|\sum_{ |\supp(\alpha)|=\ell} \tau_\alpha^{(\omega)}\cdots\tau_\alpha^{(\omega^{K-1})} a_{\alpha}z^{\alpha}\right\|_{\Omega_K^n}
				\le (2+2\sqrt{2})^{(K-1)\ell}\|f\|_{\Omega_K^n}.
			\end{align*}
			We claim that for any $\alpha$ with $|\supp(\alpha)|=\ell$, the cumulative factor introduced by iterating $\mathfrak{D}_{\omega^k}$'s,
			\begin{equation*}
				\tau_K(\ell):=\prod_{1\le k\le K-1}\tau_\alpha^{(\omega^k)}=	\prod_{1\le k\le K-1}\prod_{j:\alpha_j\neq 0}\left(1-(\omega^k)^{\alpha_j}\right),
			\end{equation*}
			is a nonzero constant depending only on $K$ and $\ell=|\supp(\alpha)|$.
			In fact, it will suffice to argue that
			\begin{equation}\label{eq:c prime}
				\prod_{1\le k\le K-1}\left(1-(\omega^k)^j\right)=:d_K
			\end{equation}
			is some nonzero constant independent of $1\le j\le K-1$.
			To see this is sufficient, recall $m_j(\alpha)=|\{k:\alpha_k=j\}|,1\le j\le K-1$ are such that
			\begin{equation*}
				\sum_{1\le j\le K-1}m_j(\alpha)=|\supp(\alpha)|=\ell.
			\end{equation*}
			Then if \eqref{eq:c prime} holds,
			we have 
			\begin{align*}
				\prod_{1\le k\le K-1}\prod_{j:\alpha_j\neq 0}\left(1-(\omega^k)^{\alpha_j}\right)
				&= \prod_{1\le k\le K-1}\prod_{1\le j\le K-1}\left(1-(\omega^k)^{j}\right)^{m_j(\alpha)}\\
				&= \prod_{1\le j\le K-1}\left[\prod_{1\le k\le K-1}\left(1-(\omega^k)^{j}\right)\right]^{m_j(\alpha)}\\
				&= \prod_{1\le j\le K-1}d_K^{m_j(\alpha)}\\
				&= d_K^\ell
			\end{align*}
			and thus the claim is shown with $\tau_K(\ell)=d_K^\ell$.
			This would yield
			\begin{align*}
				\left\| \sum_{ |\supp(\alpha)|=\ell} a_{\alpha} z^{\alpha}\right\|_{\Omega_K^n}
				\le |d_K|^{-\ell}(2+2\sqrt{2})^{(K-1)\ell}\|f\|_{\Omega_K^n},
			\end{align*}
			as desired.

			Now it remains to verify \eqref{eq:c prime}.
			We remark that until this step we have not used the assumption that $K>2$ is prime.
			Note that \eqref{eq:c prime} is nonzero since otherwise $\omega^{jk}=1$; \emph{i.e.}, $K\mid jk$ for some $1\le k\le K-1$, and this is not possible since $K>2$ is prime. Moreover, again by the primality of $K$, 
			\begin{equation*}
				\{jk: 1\le k\le K-1\}\equiv \{k:1\le k\le K-1\} \mod K, ~1\le j\le K-1.
			\end{equation*}
			Therefore, 
			\begin{equation*}
				\prod_{1\le k\le K-1}\left(1-(\omega^{k})^j\right)=\prod_{1\le k\le K-1}(1-\omega^{k})=:d_K
			\end{equation*}
			is independent of $ 1\le j\le K-1$, which finishes the proof of the claim. 
			
			Now we have shown that for the maximal support-homogeneous part $f_\ell$ of $f$:
			\begin{equation}\label{ineq:top ell}
				\|f_\ell\|_{\Om_K^n}\le C_K^d\|f\|_{\Omega_K^n}
			\end{equation}
			with $C_K=(2+2\sqrt{2})^{K-1}/|d_K|$.
			Repeating the same argument to $f-f_\ell$ whose maximal support-homogeneous part is $f_{\ell-1}$, we get by triangle inequality and \eqref{ineq:top ell} that
			\begin{equation*}
				\|f_{\ell-1}\|_{\Om_K^n}\;\leq\; C_K^d\|f-f_\ell\|_{\Om_K^n}\;
				\overset{\eqref{ineq:top ell}}{\leq}\; C_K^d(1+C_K^{d})\|f\|_{\Om_K^n}.
			\end{equation*}
			Iterating this procedure, we may obtain for all $ 0\le k\le \ell$ that
			\begin{equation}\label{ineq:top ell-k}
				\|f_{\ell-k}\|_{\Om_K^n}\le C_K^d(1+C_K^{d})^{k}\|f\|_{\Omega_K^n}
				=[(1+C_K^d)^{k+1}-(1+C_K^d)^{k}]\|f\|_{\Omega_K^n}.
			\end{equation}
			In fact, we have shown that \eqref{ineq:top ell-k} holds for $k=0,1$. Assume that \eqref{ineq:top ell-k} holds for $0\le j\le k-1$ and let us prove \eqref{ineq:top ell-k} for $k$. Since $f_{\ell-k}$ is the maximal support-homogeneous part of $f-\sum_{0\le j\le k-1}f_{\ell-j}$, we have by the previous argument proving \eqref{ineq:top ell}, the triangle inequality and the induction assumption that
			\begin{align*}
				\|f_{\ell-k}\|_{\Om_K^n}
				\; \le \; & C_K^d\left\|f-\sum_{0\le j\le k-1}f_{\ell-j}\right\|_{\Omega_K^n}\\
				\; \le \; &C_K^d\left(\|f\|_{\Omega_K^n}+\sum_{0\le j\le k-1}\|f_{\ell-j}\|_{\Omega_K^n}\right)\\
				\;\le \;&C_K^d\left(1+\sum_{0\le j\le k-1}\big[(1+C_K^d)^{j+1}-(1+C_K^d)^{j}\big]\right)\|f\|_{\Omega_K^n}\\
				\;=\; &C_K^d(1+C_K^d)^{k}\|f\|_{\Omega_K^n}.
			\end{align*}
			This finishes the proof of \eqref{ineq:top ell-k}. In particular for all $0\le j\le \ell$
			\begin{equation}
				\label{eq:induction-done}
				\|f_{j}\|_{\Om_K^n}\le C_K^d(1+C_K^{d})^{\ell-j}\|f\|_{\Omega_K^n}
				\le C_K^d(1+C_K^d)^{d}\|f\|_{\Omega_K^n}
				\le (2C^2_K)^{d^2}\|f\|_{\Omega_K^n},
			\end{equation}
			which completes the proof of the lemma.
		\end{proof}

		\section{A dimension-free inequality of Remez type}
		
		Inspired by the ideas and techniques in the last section, we may now overcome the obstacle \eqref{ineq:obstacle of convex hull}. We consider the case when $\go\ge 3$ is prime to be consistent with the proof in the last section. We refer to \cite{svz23remez} for more discussion. In the following, we shall prove that 
		
		\begin{theorem}\label{thm:remez cyclic}
			Fix $d\ge 1$ and $\go \ge 3$. For all $n\ge 1$, consider $f:\Omega_\go^n\to \C$ that is of degree at most $d$. Then $f$ extends to $\C^n$ as an analytic polynomial of degree at most $d$ and individual degree at most $\go -1$, and we have 
			\begin{equation}
				\|f\|_{\T^n}\Lesssim{d,\go}\|f\|_{\Om_\go^n}.
			\end{equation}
		\end{theorem}
		
		This is a discrete dimension-free inequality of \emph{Remez type}. 
		Consider $J$ a finite interval in $\R$ and a subset $E\subset J$ with positive Lebesgue measure $\mu(E)>0$. Let $f:\R\to \R$ be a real polynomial of degree at most $d$. A classical inequality of Remez \cite{remez36} states that 
		\begin{equation}\label{ineq:remez}
			\max_{x\in J}|f(x)|\le \left(\frac{4\mu(J)}{\mu(E)}\right)^d \max_{x\in E}|f(x)|.
		\end{equation}
		
		We divide the proof of Theorem \ref{thm:remez cyclic} into two steps. Fix $f:\Om_\go^n\to \C$ a polynomial of degree at most $d$. In the first step we show that 
		\begin{equation}
			\|f\|_{\T^n}\Lesssim{d,\go}\|f\|_{\Om_{2\go}^n},
		\end{equation}
		and in the second step we prove 
		\begin{equation}\label{ineq:2k to k}
			\|f\|_{\Om_{2\go}^n}\Lesssim{d,\go}\|f\|_{\Om_{\go}^n}.
		\end{equation}
		
		\textbf{Step 1} We shall prove 
		\begin{proposition}[Torus bounded by \unboldmath$\Omega_{2\grpord}$]
			\label{prop:Om2K-boundedness}
			Let $d,n\ge 1, K\ge 3$. Let $f:\T^n\to \C$ be an analytic polynomial of degree at most $d$ and individual degree at most $\grpord-1$.
			Then 
			\begin{equation*}
				\|f\|_{\T^n}\le C_\grpord^d\|f\|_{\Omega_{2\grpord}^n}\,,
			\end{equation*}
			where $C_\grpord\ge 1$ is a universal constant depending on $\grpord$ only.
		\end{proposition}
		
		To prove this proposition, we need the following lemma.
		
		\begin{lemma}
			\label{lem:probability measure}
			Fix $\grpord\ge 3$.
			There exists $\eps=\eps(\grpord)\in (0,1)$ such that, for all $z\in\C$ with $|z|\le \eps$, one can find a probability measure $\mu_z$ on $\Om_{2\grpord}$ such that
			\begin{equation}\label{eq:complex equations to solve}
				z^m=\E_{\xi\sim\mu_z}\xi^m, \qquad \forall \quad 0\le m\le \grpord-1\,.
			\end{equation}
		\end{lemma}
		
		\begin{proof}
			Put $\theta=2\pi/2\grpord=\pi/\grpord$ and $\omega=\omega_{2\grpord}=e^{\iu\theta}$.
			Fix a $z\in \C$. Finding a probability measure $\mu_z$ on $\Om_{2\grpord}$ satisfying \eqref{eq:complex equations to solve} is equivalent to solving 
			\begin{equation}\label{eq:real equations to solve}
				\begin{cases}
					\sum_{k=0}^{2\grpord-1}p_k=1&\\
					\sum_{k=0}^{2\grpord-1}p_k\cos(km\theta)=\Re z^m& 1\le m\le \grpord-1\\						
					\sum_{k=0}^{2\grpord-1}p_k\sin(km\theta)=\Im z^m& 1\le m\le \grpord-1
				\end{cases}
			\end{equation}
			with non-negative $p_k=\mu_z(\{\omega^k\})$ for $k=0,1,\ldots, 2\grpord-1$. Note that the $p_k$'s are non-negative and thus real.
			
			For this, it is sufficient to find a solution $\vec{p}=\vec{p}_z$ to $D_\grpord \vec p=\vec v_z$ with each entry of $\vec p=(p_0,\ldots, p_{2\grpord-1})^\top$ being non-negative. Here $D_\grpord$ is a $2\grpord\times 2\grpord$ real matrix given by
			\begin{equation*}
				D_\grpord=
				\begin{bmatrix}
					1&1 & 1 &\cdots & 1\\
					1&  \cos(\theta) &\cos(2\theta)&\cdots &\cos((2\grpord-1)\theta)\\
					\vdots & \vdots & \vdots &\vdots \\
					1&\cos(\grpord\theta) &\cos(2\grpord\theta)&\cdots &\cos((2\grpord-1)\grpord\theta)\\[0.3em]
					1&\sin(\theta) &\sin(2\theta)&\cdots &\sin((2\grpord-1)\theta)\\
					\vdots&\vdots & \vdots & \vdots &\vdots \\
					1&\sin((\grpord-1)\theta) &\sin(2(\grpord-1)\theta)&\cdots &\sin((2\grpord-1)(\grpord-1)\theta)\\
				\end{bmatrix},
			\end{equation*}
			and $v_z=(1,\Re z, \dots, \Re z^{\grpord-1},\Re z^{\grpord},\Im z,\dots, \Im z^{\grpord-1})^\top\in \R^{2\grpord}$. Note that \eqref{eq:real equations to solve} does not require the $(\grpord+1)$-th row
			\begin{equation}\label{eq:K+1 row}
				(1,\cos(\grpord\theta),\cos(2\grpord\theta),\dots, \cos((2\grpord-1)\grpord\theta))
			\end{equation}
			of $D_\grpord$.


			The matrix $D_\grpord$ is non-singular. To see this, take any $$\vec{x}=(x_0,x_1,\dots, x_{2\grpord-1})^\top\in \R^{2\grpord}$$ 
			such that $D_\grpord \vec{x}=\vec{0}$. Then 
			\begin{equation}\label{eq:zero}
				\sum_{k=0}^{2\grpord-1}(\omega^k)^m x_k=0,\qquad 0\le m\le \grpord.
			\end{equation}
			This is immediate for $0\le m\le \grpord-1$ by definition, and $m=\grpord$ case follows from the ``additional" row \eqref{eq:K+1 row} together with the fact that $\sin(k\grpord\theta)=0,0\le k\le 2\grpord-1.$
			Conjugating \eqref{eq:zero}, we get
			\begin{equation*}
				\sum_{k=0}^{2\grpord-1}(\omega^k)^m x_k=0,\qquad \grpord\le m\le 2\grpord.
			\end{equation*}
			Altogether, we have 
			\begin{equation*}
				\sum_{k=0}^{2\grpord-1}(\omega^k)^m x_k=0,\qquad 0\le m\le 2\grpord-1,
			\end{equation*}
			that is, $V\vec{x}=\vec{0}$, where $V=V_\grpord=[\omega^{jk}]_{0\le j,k\le 2\grpord-1}$ is a $2\grpord\times 2\grpord$ Vandermonde matrix given by $(1,\omega,\dots, \omega^{2\grpord-1})$. 	Since $V$ has determinant
			\begin{equation*}
				\det(V)=\prod_{0\le j<k\le 2\grpord-1}(\omega^j-\omega^k)\neq 0\,,
			\end{equation*}
			we get $\vec x=\vec 0$. So $D_\grpord$ is non-singular. 
			
			Therefore, for any $z\in\C$, the solution to \eqref{eq:real equations to solve}, thus to \eqref{eq:complex equations to solve}, is given by 
			\[
			\vec p_z=\big(p_0(z),p_1(z),\dots, p_{2\grpord-1}(z)\big)=D_\grpord^{-1}\vec v_z\in \R^{2\grpord}.
			\] 
			
			Notice one more thing about the rows of $D_\grpord$. As
			\[
			\sum_{k=0}^{2\grpord-1} (\omega^k)^m =0,\qquad m=1, 2, \dots, 2\grpord-1\,,
			\]
			we have automatically that vector $\vec p_* := \big(\frac1{2\grpord}, \dots, \frac1{2\grpord}\big)\in \R^{2\grpord}$ gives
			\[
			D_\grpord \vec p_* =(1, 0, 0, \dots, 0)^T=:\vec v_*\,.
			\]
			For any $k$-by-$k$ matrix $A$ denote 
			\[\|A\|_{\infty\to\infty}:=\sup_{\vec 0\neq v\in\R^{k}}\frac{\|Av\|_\infty}{\|v\|_\infty}.\]
			So with $\vec{p_*}:= D^{-1}_\grpord\vec{v_*}$  we have 
			\begin{align*}
				\|\vec p_z-\vec p_*\|_\infty
				&\le \|D_\grpord^{-1}\|_{\infty\to \infty}\|\vec v_z-\vec v_*\|_\infty\\
				&=\|D_\grpord^{-1}\|_{\infty\to \infty}\max\left\{\max_{1\le k\le \grpord}|\Re z^k|,\max_{1\le k\le \grpord-1}|\Im z^k|\right\}\\
				&\le \|D_\grpord^{-1}\|_{\infty\to \infty}\max\{|z|,|z|^\grpord\}.
			\end{align*}
			That is, 
			\begin{equation*}
				\max_{0\le j\le 2\grpord-1}\left|p_j(z)-\frac{1}{2\grpord}\right|
				\le \|D_\grpord^{-1}\|_{\infty\to \infty}\max\{|z|,|z|^{\grpord}\}.
			\end{equation*}
			Since $D_\grpord^{-1}\vec v_*=\vec p_*$, we have $\|D_\grpord^{-1}\|_{\infty\to \infty}\ge 2\grpord$.
			Put
			\[
			\eps_*:=\frac{1}{2\grpord\|D_\grpord^{-1}\|_{\infty\to \infty}}\in \left(0,\frac{1}{(2\grpord)^2}\right].
			\]
			Thus whenever $|z|<\eps_*<1$, we have 
			\begin{equation*}
				\max_{0\le j\le 2\grpord-1}\left|p_j(z)-\frac{1}{2\grpord}\right|
				\le |z|\|D_\grpord^{-1}\|_{\infty\to \infty}
				\le \eps_* \|D_\grpord^{-1}\|_{\infty\to \infty}\le \frac{1}{2\grpord},
			\end{equation*}
			so in particular $p_j(z)\ge 0$ for all $0\le j\le 2\grpord-1$.
		\end{proof}

		Now we are ready to prove Proposition \ref{prop:Om2K-boundedness}.
		
		\begin{proof}[Proof of Proposition \ref{prop:Om2K-boundedness}]
			Let $\eps_*$ be as in Lemma \ref{lem:probability measure}.
			With a view towards applying the lemma we begin by relating $\sup |f|$ over the polytorus to $\sup |f|$ over a scaled copy.
			Recalling that the homogeneous parts $f_k$ of $f$ are trivially bounded by $f$ over the torus: $\|f_k\|_{\T^n}\le \|f\|_{\T^n}$ (a standard Cauchy estimate). Thus we have
			\begin{equation}	\label{ineq:torus-dilation}
				\begin{split}
					\|f\|_{\T^n}\leq \sum_{k=0}^d\|f_k\|_{\T^n}\nonumber
					= \sum_{k=0}^d \eps_*^{-k}\sup_{z\in\T^n}|f_k(\eps_* z)|\nonumber
					&\leq \sum_{k=0}^d \eps_*^{-k}\sup_{z\in\T^n}|f(\eps_* z)|\nonumber\\
					&\leq (d+1)\eps_*^{-d}\sup_{z\in\T^n}|f(\eps_* z)|\nonumber
					= (d+1)\eps_*^{-d}\|f\|_{(\eps_*\T)^n}\,.
				\end{split}
			\end{equation}
			
			Let $z=(z_1,\ldots, z_n)\in (\eps_* \T)^n$.
			Then for each coordinate $j=1,2,\ldots, n$ there exists by Lemma \ref{lem:probability measure} a probability distribution $\mu_j=\mu_j(z)$ on $\Om_{2\grpord}$ for which $\E_{\xi_j\sim\mu_j}[\xi_j^k]=z_j^k$ for all $0\le k\le \grpord-1$.
			With $\mu =\mu(z):=\mu_1\times\cdots\times\mu_n$, this implies for a monomial with multi-index $\alpha\in \{0,1,\dots, \grpord-1\}^n$, $\E_{\xi\sim\mu(z)}[\xi^\alpha]=z^\alpha,$ or more generally $\E_{\xi\sim\mu(z)}[f(\xi)]=f(z)$ for $z\in(\eps_*\T)^n$ and $f$ under consideration.
			So
			\begin{equation}
				\label{eq:dilated-torus-comparison}
				\sup_{z\in(\eps_* \T)^n}|f(z)|= \sup_{z\in(\eps_* \T)^n}\Big|\E_{\xi\sim\mu(z)}f(\xi)\Big|\leq  \sup_{z\in(\eps_* \T)^n}\E_{\xi\sim\mu(z)}|f(\xi)|\leq\|f\|_{\Om_{2\grpord}^n}.
			\end{equation}
			Combining this with the  observation \eqref{ineq:torus-dilation} we conclude
			\begin{equation*}
				\|f\|_{\T^n}\leq (d+1)\eps_*^{-d}\|f\|_{(\eps_*\T)^n}\leq (d+1)\eps_*^{-d}\|f\|_{\Om_{2\grpord}^n}
				\leq C_\grpord^d\|f\|_{\Om_{2\grpord}^n}. \qedhere
			\end{equation*}
		\end{proof}
		
		\textbf{Step 2} Recall that $\omega=e^{\frac{2\pi i}{\go}}$. For convenience, we use $\sqrt{\omega}$ to denote $e^{\frac{\pi i}{\go}}$. A simple group rotation argument reduces \eqref{ineq:2k to k} to the estimate at $(\sqrt{\om},\dots, \sqrt{\om}).$
		
		\begin{proposition}\label{prop:reduction to sqrt of omega}
			To prove \eqref{ineq:2k to k} it suffices to prove 
			\begin{equation}\label{ineq:k to sqrt omega}
				|f(\sqrt{\om},\dots, \sqrt{\om})|\Lesssim{d,\go} \|f\|_{\Om_\go^n}
			\end{equation}
			for all $f:\T^n\to \C$ of degree at most $d$ and individual degree at most $\go-1$ for all $n\ge 1$.
		\end{proposition}
		
		\begin{proof}
			The proof is left as an exercise and we refer to \cite{svz23remez} for details. Note that $\Omega_{2\go}=\Omega_{\go}\times \{1,\sqrt{\omega}\}$.
		\end{proof}
		
		
		The proof of Proposition \ref{prop:reduction to sqrt of omega} is the subject of the rest of this subsection. 
		Our approach is to split $f$ into parts $f=\sum_{j}g_j$ such that each part $g_j$ has the properties A and B:
		
		\begin{equation}\label{ineq:proof ideas}
			\|f\|_{\Om_\grpord^n}
			\overset{\text{Property A}}{\Gtrsim{d,\grpord}}
			\|g_j\|_{\Om_\grpord^n}
			\overset{\text{Property B}}{\Gtrsim{d,\grpord}}
			|g_j(\sqrt{\om},\dots, \sqrt{\om})|\,.
		\end{equation}
		
		Such splitting gives
		\[|f(\sqrt{\om},\dots, \sqrt{\om})|\leq\sum_j|g_j(\sqrt{\om},\dots, \sqrt{\om})|\Lesssim{d,\grpord}\sum_j\|g_j\|_{\Om_\grpord^n}\Lesssim{d,\grpord}\sum_j\|f\|_{\Om_\grpord^n}\,.\]
		So as long as the number of $g_j$'s is independent of $n$ such a splitting with Properties A and B entails the result.

		For this we shall employ the operator $\mathfrak{D}=\mathfrak{D}_{\omega}$ in the last section that is defined via 
		\begin{equation}
			\mathfrak{D}f(z)= \sum_{ |\supp(\alpha)|=L} \tau_\alpha a_{\alpha} z^{\alpha}\qquad \textnormal{ for }\qquad f(z)= \sum_{ |\supp(\alpha)|\le L} a_{\alpha} z^{\alpha},
		\end{equation}
		where $\tau_{\alpha}:=\tau_{\alpha}^{(\omega)}=\prod_{j:\alpha_j\neq 0}(1-\omega^{\alpha_j}).$ Inspired by $\mathfrak{D}$, we say that two monomials $m=z^\alpha, m'=z^\beta$ are \emph{inseparable} if $|\supp(\alpha)|=|\supp(\beta)|$ and $\tau_{\alpha}=\tau_{\beta}$. When two monomials $m,m'$ are inseparable, we write $m\sim m'.$
		
		
		\begin{proposition}[Property A]\label{prop:Property A}
			Fix $\grpord\ge 3$ and $d\ge 1$. Suppose that $f:\Om_\grpord^n\to\C$ is a polynomial of degree at most $d$ with maximum support size $L$. For $0\leq\ell\leq L$ let $f_\ell$ denote the part of $f$ composed of monomials of support size $\ell$, and let $g_{(\ell, 1)},\ldots, g_{(\ell, J_\ell)}$ be the inseparable parts of $f_\ell$.
			Then there exists a universal constant $C_{d,\grpord}$ independent of $n$ and $f$ such that for all $0\leq \ell\leq L$ and $1\leq j\leq J_\ell$,
			\[\|g_{(\ell, j)}\|_{\Om_\grpord^n}\leq C_{d,\grpord}\|f\|_{\Om_\grpord^n}\,.\]
		\end{proposition}
		\begin{proof}
			We first show the proposition for $g_{(L,j)}$, $1\leq j \leq J_L$. Suppose that 
			\begin{equation*}
				f(z)=\sum_{\alpha:|\supp(\alpha)|\le L}a_{\alpha}z^{\alpha}.
			\end{equation*}
			Inductively, one obtains from Proposition \ref{prop:bdd of d} that for $1\leq k\leq J_L$,
			\begin{equation}\label{ineq:c(alpha)^k}
				\mathfrak{D}^{k}f= \sum_{ |\supp(\alpha)|=L} \tau_\alpha^ka_{\alpha} z^{\alpha}
				\qquad\text{ with }\qquad\left\| \mathfrak{D}^{k}f\right\|_{\Om_\grpord^n}\le (2+2\sqrt{2})^{kL}\|f\|_{\Om_\grpord^n}.
			\end{equation}
			
			By definition there are $J_L$ distinct values of $\tau_\alpha$ among 
			the monomials of $f_L$; label them $c_1,\ldots, c_{J_L}$.
			Then
			\begin{equation*}
				f_{L}(z)=\sum_{ |\supp(\alpha)|=L} a_{\alpha} z^{\alpha}=\sum_{1\le j\le J_L}g_{(L,j)}(z), \quad\text{and}
			\end{equation*}
			\begin{equation*}
				\mathfrak{D}^{k}f(z)=\sum_{ |\supp(\alpha)|=L} \tau_\alpha^k a_{\alpha}z^{\alpha}=\sum_{1\le j\le J_L}c_j^k g_{(L,j)}(z),\qquad k\ge 1.
			\end{equation*}
			Let us confirm $J_L$ is independent of $n$.
			Consider $\alpha$ with $|\supp(\alpha)|=L$.
			We may count the support size of $\alpha$ by binning coordinates according to their degree:
			$|\supp(\alpha)|=L$,
			\[\sum_{1\leq t\leq \grpord-1}|\{s\in[n]:\alpha_s=t\}|=L\leq d,\]
			so 
			\begin{align}
				\label{ineq:J_L}
				\begin{split}
					J_L &\leq |\{(m_1,\dots, m_{\grpord-1})\in \{0,\ldots, L\}^{\grpord-1}:m_1+\cdots  +m_{\grpord-1}= L\}|\\
					&\le  \binom{\grpord-1+L-1}{L-1}\leq (\grpord+d)^d\,.
				\end{split}
			\end{align}
			
			According to \eqref{ineq:c(alpha)^k}, we have
			
			\begin{equation*}
				\begin{pmatrix}
					\mathfrak{D}f\\
					\mathfrak{D}^{2}\!f\\
					\vdots \\
					\mathfrak{D}^{J_L}\!f\\
				\end{pmatrix}
				=\underbracedmatrix{\begin{matrix}
						c_1 & c_2 & \cdots & c_{J_L}\\
						c_1^2 & c_2^2 &\cdots & c_{J_L}^2\\
						\vdots & \vdots & \ddots & \vdots\\
						c_1^{J_L} & c_2^{J_L} & \cdots & c_{J_L}^{J_L}
				\end{matrix}}{=:\; V_L}
				\begin{pmatrix}
					g_{(L,1)}\\
					g_{(L,2)}\\
					\vdots \\
					g_{(L,J_L)}\\
				\end{pmatrix}.
			\end{equation*}
			
			The $J_L\times J_L$ modified Vandermonde matrix $V_L$ has determinant
			\[\det(V_L) = \left(\prod_{j=1}^{J_L} c_j\right) \left( \prod_{1\leq s < t \leq J_L}(c_s-c_t)\right).\]
			Since the $c_j$'s are distinct and nonzero we have $\det(V_L)\neq 0$. So $V_L$ is invertible and in particular $g_{(L,j)}$ is the $j$\textsuperscript{th} entry of $V_L^{-1}(\mathfrak{D}^1\!f,\ldots,\mathfrak{D}^{J_L}\!f)^\top$.
			Letting $\eta^{(L,j)}=(\eta^{(L,j)}_k)_{1\le k\le J_L}$ be the $j$\textsuperscript{th} row of $V_L^{-1}$, this means
			\begin{align*}
				g_{(L,j)}&=\sum_{1\le k\le J_L}\eta_k^{(L,j)}\, \mathfrak{D}^{k}\!f.
			\end{align*}
			And $\eta^{(L,j)}$ depends on $d$ and $\grpord$ only, so for all $1\leq j\leq J_L$,
			\begin{equation}
				\label{eq:top-supp-bd}
				\|g_{(L,j)}\|_{\Om_\grpord^n}\le \sum_{1\le k\le J_L} \big|\eta_k^{(L,j)}\big|\!\cdot\!	\left\|\mathfrak{D}^{k}\!f\right\|_{\Om_\grpord^n}
				\le \|\eta^{(L,j)}\|_1\big(2+2\sqrt{2}\big)^{J_Ld}	\|f\|_{\Om_\grpord^n},
			\end{equation}
			where we used \eqref{ineq:c(alpha)^k} in the last inequality. The constant \[\|\eta^{(L,j)}\|_1(2+2\sqrt{2})^{J_Ld}\leq C(d,\grpord)<\infty\]
			for appropriate $C(d,\grpord)$ is dimension-free and depends only on $d$ and $\grpord$ only.
			This finishes the proof for the inseparable parts in $f_L$.
			
			We now repeat the argument on $f-f_L$ to obtain \eqref{eq:top-supp-bd} for the inseparable parts of support size $L-1$.
			In particular, there are vectors $\eta^{(L-1,j)}$, $1\leq j\leq J_{L-1}$ of dimension-free 1-norm with
			\[\|g_{(L-1,j)}\|_{\Om_\grpord^n}\leq C(d,\grpord)\|\eta^{(L-1,j)}\|_1\|f-f_L\|_{\Om_\grpord^n}
			\Lesssim{d,\grpord}\|f-f_L\|_{\Om_\grpord^n}\,.\]
			This can be further repeated to obtain for $0\leq \ell\leq L$ and $1\leq j\leq J_\ell$, the vectors $\eta^{(\ell,j)}$ with dimension-free 1-norm such that
			\[\|g_{(\ell, j)}\|_{\Om_\grpord^n}\Lesssim{d,\grpord}\left\|f-\sum_{\ell+1\leq k\leq L}f_k\right\|_{\Om_\grpord^n}\,.\]
			
			It remains to relate $\|f-\sum_{\ell+1\leq k \leq L}f_k\|_{\Om_\grpord^n}$ to $\|f\|_{\Om_\grpord^n}$.
			Note that with $V_L$ as originally defined, by considering $(1\,1\,\dots\,1)V_L^{-1}(\mathfrak{D}^1f,\ldots, \mathfrak{D}^{J_L}f)^\top$ we obtain a constant $D_L=D_L(d,\grpord)$ independent of $n$ for which
			\[\|f_L\|_{\Om_\grpord^n}\leq D_L\|f\|_{\Om_\grpord^n}\,.\]
			This means
			\[\|f-f_L\|_{\Om_\grpord^n}\leq (1+D_L)\|f\|_{\Om_\grpord^n}\,.\]
			Notice the top-support part of $f-f_L$ is exactly $f_{L-1}$, so repeating the argument above on $f-f_L$ yields a constant $D_{L-1}=D_{L-1}(d,\grpord)$ such that
			\begin{equation}
				\|f_{L-1}\|_{\Om_\grpord^n}\leq D_{L-1}\|f-f_L\|_{\Om_\grpord^n}
				\le D_{L-1}(1+D_L)\|f\|_{\Om_\grpord^n}=(D_{L-1}+D_{L-1}D_L)\|f\|_{\Om_\grpord^n}\,.
			\end{equation}
			Continuing, for $1\leq \ell \leq L$ we find 
					\begin{align*}
						\|f_{L-\ell}\|_{\Om_\grpord^n}
						\leq D_{L-\ell}\left\|f-\sum_{L-\ell+1\le k\le L}f_k\right\|_{\Om_\grpord^n}
						&\leq D_{L-\ell}(1+D_{L-\ell+1})\left\|f-\sum_{L-\ell+2\le k\le  L}f_k\right\|_{\Om_\grpord^n}\\
						&\leq \cdots\leq D_{L-\ell}\prod_{0\le k\le \ell-1}(1+D_{L-k})\|f\|_{\Om_\grpord^n}.
					\end{align*}
					We have found for each $\ell$-support-homogeneous part of $f$,
					\[\|f_\ell\|_{\Om_\grpord^n}\Lesssim{d,\grpord}\|f\|_{\Om_\grpord^n},\]
					so we have $\|f-\sum_{\ell+1\leq k \leq L}f_k\|_{\Om_\grpord^n}\Lesssim{d,\grpord}\|f\|_{\Om_\grpord^n}$ as well.
				\end{proof}
				
				\subsubsection{Property B: Boundedness at \texorpdfstring{$(\sqrt{\omega},\dots, \sqrt{\omega})$}{√ω} for inseparable parts}
				
				Here we argue $g(\sqrt{\omega},\dots, \sqrt{\omega})$ is bounded for inseparable $g$. Recall that $\omega=e^{\frac{2\pi \iu}{K}}$ and $\sqrt{\omega}=e^{\frac{\pi \iu}{K}}$.
				\begin{proposition}[Property B]
					\label{prop:Property B}
					If $g$ is a linear combination of inseparable monomials,  then 
					\begin{equation}
						g(\sqrt{\omega},\dots, \sqrt{\omega})|\leq\|g\|_{\Om_\grpord^n}.
					\end{equation}
				\end{proposition}
				
				\begin{proof}
					We will need an identity for half-roots of unity.
					For $k=1,\ldots, \grpord-1$ we have
					\begin{equation}
						\label{eq:trig-id}
						(\sqrt{\omega})^k=\iu\frac{1-\omega^k}{|1-\omega^k|}\,,
					\end{equation}
					following from the orthogonality of $(\sqrt{\omega})^k$ and $1-\omega^k$ in the complex plane.
					
					We claim that for two monomials $m=z^\alpha$ and $m'=z^{\beta}$
					\begin{equation*}
						m\sim m' \;\implies\; m(\sqrt{\om},\dots, \sqrt{\om})=m'(\sqrt{\om},\dots, \sqrt{\om})\,.
					\end{equation*}
					By definition $m\sim m'$ means $m$ and $m'$ have the same support size (call it $\ell$) and
					\[\textstyle\prod_{j:\alpha_j\neq0}(1-\omega^{\alpha_j}) = \prod_{j:\beta_j\neq0}(1-\omega^{\beta_j})\,.\]
					Dividing both sides by the modulus and multiplying by $\iu^\ell$ allows us to apply \eqref{eq:trig-id} to find
					\[\textstyle\prod_{j:\alpha_j\neq0}(\sqrt{\omega})^{\alpha_j}=\prod_{j:\beta_j\neq0}(\sqrt{\omega})^{\beta_j}\,,\]
					which is exactly $m(\sqrt{\om},\dots, \sqrt{\om})=m'(\sqrt{\om},\dots, \sqrt{\om}).$
					
					Now let $\zeta=m(\sqrt{\omega})\in \T$ for some monomial $m$ in $g$.
					Then because $\zeta$ is independent of $m$, with $g=\sum_{\alpha\in S}a_\alpha z^\alpha$, we have $g(\sqrt{\om},\dots, \sqrt{\om})=\zeta\sum_{\alpha\in S}a_\alpha$ and
					\begin{equation*}
						\textstyle
						|g(\sqrt{\om},\dots, \sqrt{\om})| =|\sum_{\alpha\in S}a_\alpha| = |g(\vec1)|\leq \|g\|_{\Om_\grpord^n}\,.\qedhere
					\end{equation*}
				\end{proof}

				Now we may conclude the proof of Theorem \ref{thm:remez cyclic}.
				
				\begin{proof}[Proof of Theorem \ref{thm:remez cyclic}]
					It remains to prove \textbf{Step 2} estimate \eqref{ineq:2k to k}. With the help of Proposition \ref{prop:reduction to sqrt of omega}, it suffices to show  \eqref{ineq:k to sqrt omega}. 
					Write $f=\sum_{0\leq \ell \leq L}\sum_{1\leq j\leq J_\ell}g_{(\ell,j)}$ in terms of inseparable parts, where $g_{(\ell,j)},1\le j\le J_\ell,0\le \ell \le L$ are as in Proposition \ref{prop:Property A}. Then by Propositions \ref{prop:Property A} (Property A) and \ref{prop:Property B} (Property B)
					\begin{align*}
						|f(\sqrt{\omega},\dots, \sqrt{\omega})|\;&\leq\quad\sum_{0\leq \ell \leq L}\sum_{1\leq j\leq J_\ell}|g_{(\ell,j)}(\sqrt{\omega},\dots, \sqrt{\omega})|\\
						&\leq\quad\sum_{0\leq \ell \leq L}\sum_{1\leq j\leq J_\ell}\|g_{(\ell,j)}\|_{\Om_\grpord^n}\tag{Property B}\\
						&\Lesssim{d,\grpord}\|f\|_{\Om_\grpord^n}\sum_{0\leq \ell \leq L}J_\ell\tag{Property A}\,.
					\end{align*}
					In view of \eqref{ineq:J_L} and $L\le d$, we obtain $|f(\sqrt{\omega},\dots, \sqrt{\omega})|\Lesssim{d,\grpord}\|f\|_{\Om_\grpord^n}$ which is exactly  \eqref{ineq:k to sqrt omega}.  
				\end{proof}
				
\newcommand{\etalchar}[1]{$^{#1}$}

			\end{document}